\documentclass[a4paper,english,11pt,oldtoc]{artikel3}
\usepackage[english]{babel}
\usepackage[utf8]{inputenc}
\usepackage[T1]{fontenc}
\usepackage[english]{babel}
\usepackage[labelfont=bf]{caption}
\usepackage[labelfont=bf]{subcaption}
\usepackage{comment}
\usepackage{footmisc}
\usepackage{bbm}
\usepackage{amsmath}
\usepackage{amssymb}
\usepackage{amsthm}
\usepackage{ifpdf}

\ifpdf
    \usepackage[pdftex]{hyperref}
\else
    \usepackage[hypertex]{hyperref}
\fi

\usepackage[svgnames]{xcolor}
\definecolor{hrefcolor}{rgb}{0.0,0.5,0.8}
\definecolor{hlgreen}{rgb}{0,0.7,0}

\hypersetup{
   colorlinks=true,
   linkcolor=hrefcolor,
   citecolor=hlgreen,
   filecolor=hrefcolor,
   urlcolor=hrefcolor,
}

\ifdefined\tikz
\else
    \ifpdf
        \usepackage[pdftex]{graphicx}
    \else
        \usepackage{graphicx}
        
    \fi

\fi

\newcounter{remcount}

\newtheorem{theorem}{Theorem}

\newtheorem{lemma}{Lemma}
\newtheorem{proposition}{Proposition}

\theoremstyle{definition}

\newtheorem*{assumption*}{Assumption}
\newtheorem{remark}{Remark}
\newtheorem*{remark*}{Remark}
\newtheorem*{definition*}{Definition}
\newtheorem{example}{Example}

\numberwithin{equation}{section}
\numberwithin{lemma}{section}
\numberwithin{theorem}{section}
\numberwithin{proposition}{section}
\numberwithin{definition}{section}
\numberwithin{remark}{section}
\numberwithin{example}{section}
\numberwithin{assumption}{section}
\numberwithin{algorithm}{section}
\numberwithin{corollary}{section}

\newcommand*{\doi}[1]{doi:\href{http://dx.doi.org/#1}{\detokenize{#1}}}

\makeatletter
\AtBeginDocument{
    \hypersetup{
        pdftitle = {\@title},
        pdfauthor = {\@author}
    }
}
\makeatother

\def\email#1{\texttt{\href{mailto:#1}{#1}}}

\usepackage{tikz}
\usepackage{float}
\usepackage{algorithmic}
\usepackage[numbers]{natbib}
\usepackage{stmaryrd}

\newcommand{\term}{\emph}

\newcommand{\field}[1]{\mathbb{#1}}
\newcommand{\N}{\mathbb{N}}

\newcommand{\R}{\field{R}}
\newcommand{\B}{B}

\newcommand{\norm}[1]{\|#1\|}

\newcommand{\abs}[1]{|#1|}

\newcommand{\inv}[1]{#1^{-1}}
\newcommand{\grad}{\nabla}
\newcommand{\freevar}{\,\boldsymbol\cdot\,}

\newcommand{\Union}\bigcup
\newcommand{\Isect}\bigcap
\newcommand{\union}\cup
\newcommand{\isect}\cap
\newcommand{\bigunion}\bigcup
\newcommand{\bigisect}\bigcap

\newcommand{\defeq}{:=}

\newcommand{\downto}{\searrow}

\newcommand{\subdiff}{\partial}

\DeclareMathOperator{\Dom}{dom}

\makeatletter
\def \uminus@sym{\setbox0=\hbox{$\cup$}\rlap{\hbox 
        to\wd0{\hss\raise0.5ex\hbox{$\scriptscriptstyle{-}$}\hss}}\box0}
    \def \uminus    {\mathrel{\uminus@sym}}
\makeatother

\newcommand{\mathvar}[1]{\textup{#1}}

\renewcommand{\tilde}{\widetilde}

\newcommand{\iprod}[2]{\langle #1,#2\rangle}

\newcommand{\BVspace}{\mathvar{BV}}

\newcommand{\Eabs}{\mathcal{E}}

\renewcommand{\L}{\mathcal{L}}

\renewcommand{\d}{\,d} 

\newcommand{\TGV}{\mathvar{TGV}}
\newcommand{\TV}{\mathvar{TV}}

\makeatletter
\def \weaktostar@sym{\setbox0=\hbox{$\rightharpoonup$}\rlap{\hbox 
        to\wd0{\hss\raise1ex\hbox{$\scriptscriptstyle{*\,}$}\hss}}\box0}
    \def \weaktostar    {\mathrel{\weaktostar@sym}}
\makeatother

\usepackage[textsize=footnotesize,color=DarkBlue!40]{todonotes}

\floatstyle{ruled}
\newfloat{Algorithm}{t}{loa}

\def\extR{\overline \R}
\def\linear{\mathcal{L}}

\newcommand{\linearLArrow}[1][]{\linear_{\triangleleft\ifx&#1&\else,\,#1\fi}}
\newcommand{\linearLArrowSpecial}[1][]{\linear^{\star}_{\triangleleft\ifx&#1&\else,\,#1\fi}}

\def\realopt#1{\widehat #1}
\def\this#1{#1^i}
\def\next#1{#1^{i+1}}
\def\overnext#1{\bar #1^{i+1}}

\def\realoptu{{\realopt{u}}}
\def\realoptx{{\realopt{x}}}

\def\realopty{{\realopt{y}}}
\def\nextu{\next{u}}
\def\nextx{\next{x}}

\def\nexty{\next{y}}

\def\thisu{\this{u}}
\def\thisx{\this{x}}

\def\thisy{\this{y}}

\def\overnextx{\overnext{x}}

\def\OverRelax{\Omega}

\def\invstar#1{{#1}^{-1,*}}

\def\Tau{T}
\def\TauHat{{\hat \Tau}}
\def\TauTilde{{\tilde \Tau}}
\def\TauSpace{\L(X; X)}
\def\TauHatSpace{\hat{\mathcal{\Tau}}}
\def\TauTildeSpace{\tilde{\mathcal{\Tau}}}
\def\UHatSet{\hat{\mathcal{K}}}
\def\UTildeSet{\tilde{\mathcal{K}}}

\def\gap{\mathcal{G}}

\def\ergGap#1{\gap^{#1}}

\makeatletter
\AtBeginDocument{
\hypersetup{
    pdftitle = {Acceleration of the PDHGM on strongly convex subspaces},
    pdfauthor = {T.~Valkonen, T.~Pock}
}
}
\makeatother

\begin{document}

\title{Acceleration of the PDHGM on strongly convex subspaces}
\date{\today}
\author{
    Tuomo Valkonen\thanks{Department of Applied Mathematics and Theoretical Physics, University of Cambridge, United Kingdom. \email{tuomo.valkonen@iki.fi}}
    \thanks{Department of Mathematical Sciences, University of Liverpool, United Kingdom}
    \and
    Thomas Pock\thanks{Institute for Computer Graphics and Vision, Graz University of Technology, 8010 Graz, Austria. Digital Safety \& Security Department, AIT Austrian Institute of Technology GmbH, 1220 Vienna, Austria. \email{pock@icg.tugraz.at}}
    }

\maketitle

\begin{abstract}
    We propose several variants of the primal-dual method due to Chambolle and Pock.
    Without requiring full strong convexity of the objective functions, our methods are accelerated on subspaces with strong convexity.
    This yields mixed rates, $O(1/N^2)$ with respect to initialisation and $O(1/N)$ with respect to the dual sequence, and the residual part of the primal sequence. 
    We demonstrate the efficacy of the proposed methods on image processing problems lacking strong convexity, such as total generalised variation denoising and total variation deblurring.
\end{abstract}


\section{Introduction}

Let $G: X \to \extR$ and $F^*: Y \to \extR$ be convex, proper, and lower semicontinuous functionals on Hilbert spaces $X$ and $Y$, possibly infinite-dimensional. Also let $K \in \linear(X; Y)$ be a bounded linear operator. We then wish to solve the minimax problem
\begin{equation}
    \notag
    \min_{x \in X} \max_{y \in Y} \ G(x) + \iprod{K x}{y} - F^*(y).
\end{equation}
One possibility is the primal-dual algorithm of Chambolle and Pock \cite{chambolle2010first}, a type of proximal point or extragradient method, also classified as the ``modified primal-dual hybrid gradient method'' or PDHGM by Esser \cite{esser2010general}. If either $G$ of $F^*$ is strongly convex, the method can be accelerated to produce Nesterov's \cite{nesterov1983method} optimal $O(1/N^2)$ rates. But what if we have only partial strong convexity? For example, what if 
\[
    G(x)=G_0(P x)
\]
for a projection operator $P$ to a subspace $X_0 \subset X$, and strongly convex $G_0: X^0 \to \R$? This kind of structure is common in many applications in image processing and the data sciences, as we will more closely review in Section \ref{sec:example}. Under such \term{partial strong convexity}, can we obtain a method that would give an accelerated rate of convergence at least for $P x$?

We provide a partially positive answer: we can obtain mixed rates, $O(1/N^2)$ with respect to initialisation, and $O(1/N)$ with respect to bounds on the ``residual variables'' $y$ and $(I-P)x$.
In this, our results are similar to the ``optimal'' algorithm of Chen et al.~\cite{chen2015optimal}. Instead of strong convexity, they assume smoothness of $G$ to derive a primal-dual algorithm based on backward--forward steps, instead of the backward--backward steps of \cite{chambolle2010first}. 

The derivation of our algorithms is based, firstly, on replacing simple step length parameters by a variety of abstract step length operators and, secondly, a type of abstract partial strong monotonicity property
\begin{equation}
    \label{eq:mono-intro}
    \iprod{\subdiff G(x') - \subdiff G(x)}{\inv\TauTilde(x'-x)} \ge 
        \norm{x'-x}_{\invstar\TauTilde\Gamma'}^2 - \psi_{\invstar\TauTilde(\Gamma'-\Gamma)}({x'-x}),
\end{equation}
the full details of which we provide in Section \ref{sec:abstract}.
In this, we make the monotonicity dependent on the step length operator $\TauTilde$.
%
Secondly, our factor of strong convexity is the operator $\Gamma$, which is however shifted in \eqref{eq:mono-intro} into a \term{penalty term} $\psi$ through the introduction of additional strong monotonicity in terms of $\Gamma' \ge \Gamma$.
This exact procedure can be seen as a type of smoothing, famously studied by Nesterov \cite{nesterov2005smooth}, and more recently, for instance, by Beck and Teboulle \cite{beck2012smoothing}. In these approaches, one computes \emph{a priori} a level of smoothing---comparable to $\Gamma'$---needed to achieve certain quality of solution, and then solves a smoothed problem at the optimal $O(1/N^2)$ rate. However, to achieve a better solution than the a priori chosen quality, one needs to solve a new problem from scratch, or to develop restarting strategies. Our approach does not depend on restarting and a priori chosen solution qualities. Indeed, $\Gamma'$ is controlled automatically. In most applications, $\psi_{\invstar\TauTilde(\Gamma'-\Gamma)}({x'-x})=\inv{\tilde\tau}\gamma^\perp C$ for $\gamma^\perp$ the introduced strong monotonicity on the orthogonal complement $X_0^\perp$. This kind of constant $\psi$ can in particular be achieved on bounded domains, as was also employed for the aforementioned mixed-rate algorithm \cite{chen2015optimal}.

The ``fast dual proximal gradient method'', or FDPG \cite{beck2014fdpg}, also possesses different type of mixed rates, $O(1/N)$ for the primal, and $O(1/N^2)$ for the dual. This is however under standard strong convexity assumptions.
Other than that, our work is related to various further developments form the PDHGM, such as variants for non-linear $K$ \cite{tuomov-nlpdhgm,benning2015preconditioned}, and non-convex $G$ \cite{mollenhoff2014primal}.
It has been the basis for inertial methods for monotone inclusions \cite{lorenz2014accelerated}, and primal-dual stochastic coordinate descent methods without separability requirements \cite{fercoq2015coordinate}. Finally, the FISTA \cite{beck2009fista,beck2009fast} can be seen as a primal-only version of the PDHGM.
Not attempting to do full justice here to the large family of closely-related methods, we point to \cite{esser2010general,setzer2011operator,tuomov-big-images} for further references.

The contributions of our paper are twofold: firstly, to paint a bigger picture of what is possible, we derive a very general version of the PDHGM. This algorithm, useful as a basis for deriving other new algorithms besides ours, is the content of Section \ref{sec:abstract}.
A byproduct of this work is the shortest convergence rate proof for the accelerated PDHGM known to us.
Secondly, in Section \ref{sec:special}, we derive from the general algorithm two efficient mixed-rate algorithms for problems exhibiting strong convexity only on subspaces. The first one employs the penalty or smoothing $\psi$ on both the primal and the dual. The second one only employs the penalty on the dual. We do some of the groundwork for these algorithms in Section \ref{sec:scalar0}. We finish the study with numerical experiments in Section \ref{sec:example}. The main results of interest for readers wishing to apply our work are Algorithms \ref{alg:alg-projective-both} and \ref{alg:alg-proj} along with the respective convergence results, Theorem \ref{thm:projective-both} and Theorem \ref{thm:projective-dualonly}.

\section{A general primal-dual method}
\label{sec:abstract}

\subsection{Background}

As in the introduction, let us be given convex, proper, lower semicontinuous functionals
$G: X \to \extR$ and $F^*: Y \to \extR$ on Hilbert spaces $X$ and $Y$, as well as a bounded linear operator $K \in \linear(X; Y)$.
We then wish to solve the minimax problem
\begin{equation}
    \label{eq:problem}
    \tag{P}
    \min_{x \in X} \max_{y \in Y} \ G(x) + \iprod{K x}{y} - F^*(y),
\end{equation}
assuming the existence of a solution $\realoptu=(\realoptx, \realopty)$ satisfying the optimality conditions
\begin{equation}
    \label{eq:oc}
    \tag{OC}
    -K^* \realopty \in \subdiff G(\realoptx),
    \quad\text{and}\quad
    K \realoptx \in \subdiff F^*(\realopty).
\end{equation}
Such a point always exists if $\lim_{\norm{x} \to \infty} G(x)=\infty$ and $\lim_{\norm{y} \to \infty} F^*(y)=\infty$ \cite[Proposition VI.1.2 \& Proposition VI.2.2]{ekeland1999convex}. More generally the existence has to be proved explicitly.
In finite dimensions, see, e.g., \cite{rockafellar-wets-va} for sufficient conditions.

The primal-dual method of Chambolle and Pock \cite{chambolle2010first} for \eqref{eq:problem} consists of iterating the system
\begin{subequations}
\label{eq:cp}
\begin{align}
    \nextx & \defeq (I+\tau_i \subdiff G)^{-1}(\thisx - \tau_i K^* y^{i}),\\
    \label{eq:cp-overnextx}
    \overnextx & \defeq \omega_i (\nextx-\thisx)+\nextx, \\
    \nexty & \defeq (I+\sigma_{i+1} \subdiff F^*)^{-1}(\thisy + \sigma_{i+1} K \overnextx).
\end{align} 
\end{subequations}
In the basic version of the algorithm, $\omega_i=1$, $\tau_i \equiv \tau_0$, and $\sigma_i \equiv \sigma_0$, assuming that the step length parameters satisfy $\tau_0 \sigma_0 \norm{K}^2 < 1$. The method has $O(1/N)$ rate for the ergodic duality gap. If $G$ is strongly convex with factor $\gamma$, we may accelerate
\begin{equation}
    \label{eq:cpaccel}
    \omega_i \defeq 1/\sqrt{1+2\gamma\tau_i},
    \quad
    \tau_{i+1} \defeq \tau_i\omega_i,
    \quad\text{and}\quad
    \sigma_{i+1} \defeq \sigma_i/\omega_i,
\end{equation}
to achieve $O(1/N^2)$ convergence rates.
To motivate our choices later on, observe that $\sigma_0$ is never needed, if we equivalently parametrise the algorithm by $\delta = 1 - \norm{K}^2 \tau_0\sigma_0 > 0$, which gives the fixed ratio of $\sigma_i$ to $\tau_i$.

We note that the order of the steps in \eqref{eq:cp} is reversed from the original ordering in \cite{chambolle2010first}. This is because with the present order, the method \eqref{eq:cp} may also be written in the proximal point form. This formulation, first observed in \cite{he2012convergence} and later utilised in \cite{tuomov-nlpdhgm,pock2011iccv,benning2015preconditioned}, is also what we will use to streamline our analysis. Introducing the general variable splitting notation,
\[
    u=(x, y),
\]
the system \eqref{eq:cp} then reduces into
\begin{equation}
    \label{eq:pp}
    0 \in H(\nextu) + M_{\text{basic},i}(\nextu-\thisu),
\end{equation}
for the monotone operator
\begin{equation}
    \label{eq:h}
    H(u) \defeq
        \begin{pmatrix}
            \subdiff G(x) + K^* y \\
            \subdiff F^*(y) -K x
        \end{pmatrix},
\end{equation}
and the \emph{preconditioning} or \emph{step-length operator}
\begin{equation}
    \label{eq:cp-m}
    M_{\text{basic},i} \defeq
        \begin{pmatrix}
            1/\tau_i & -K^* \\
            -\omega_i K & 1/\sigma_{i+1}
        \end{pmatrix}.
\end{equation}
We note that the optimality conditions \eqref{eq:oc} can also be encoded as $0 \in H(\realoptu)$.

\subsection{Abstract partial monotonicity}

Our plan now is to formulate a general version of \eqref{eq:cp}, replacing $\tau_i$ and $\sigma_i$ by operators $\Tau_i \in \L(X; X)$ and $\Sigma_i \in \L(Y; Y)$. In fact, we will need two additional operators $\TauTilde_i \in \L(X; X)$ and $\TauHat_i \in \L(Y; Y)$ to help communicate change in $\Tau_i$ to $\Sigma_i$.
They replace $\omega_i$ in \eqref{eq:cp-overnextx} and \eqref{eq:cp-m}, operating as $\TauHat_{i+1} K \inv\TauTilde_i \approx \omega_i K$ from both sides of $K$.
The role of $\TauTilde_i$ is to split the primal step length in space the $X$ into two parts with potentially different rates, $\Tau_i$ and $\TauTilde_i$, while $\TauHat_i$ transfers $\TauTilde_i$ into the space $Y$, to eventualy control the dual step length $\Sigma_i$.
In the basic algorithm \eqref{eq:cp}, we would simply have $\TauTilde_i=\Tau_i=\tau_i I \in \L(X; X)$, and $\TauHat_i=\tau_i I \in \L(Y; Y)$ for the scalar $\tau_i$.


To make the notation definite, we denote by $\linear(X; Y)$ the space of bounded linear operators between Hilbert spaces $X$ and $Y$. For $T \in \linear(X; X)$, the notation $T \ge 0$ means that $T$ is positive semidefinite. In this case, we also denote
\[
    [0, T] \defeq \{\lambda T \mid \lambda \in [0, 1]\}.
\]
For $M \in \L(X; X)$, which can possibly not be self-adjoint, we employ the notation
\begin{equation}
    \label{eq:iprod-def}
    \iprod{a}{b}_M \defeq \iprod{Ma}{b},
    \quad
    \text{and}
    \quad
    \norm{a}_M \defeq \sqrt{\iprod{a}{a}_M}.
\end{equation}
We also use the notation $\invstar T \defeq (\inv T)^*$.

To start the algorithm derivation, we now formulate abstract forms of partial strong monotonicity. As a first step, we take subspaces of invertible operators
\[
    \TauTildeSpace \subset \linear(X; X),
    \quad\text{and}\quad
    \TauHatSpace \subset \linear(Y; Y),
\]
as well as subsets of positive semidefinite operators
\[
    0 \le \UTildeSet \subset \linear(X; X),
    \quad\text{and}\quad
    0 \le \UHatSet \subset \linear(Y; Y).
\]
We assume $\TauTildeSpace$ and $\TauHatSpace$ closed with respect to composition: $\TauTilde_1\TauTilde_2 \in \TauTildeSpace$ for $\TauTilde_1,\TauTilde_2 \in \TauTildeSpace$.

We use the sets $\UTildeSet$ and $\UHatSet$ as follows.
We suppose that $\subdiff G$ is \emph{partially strongly $(\psi, \TauTildeSpace, \UTildeSet)$-monotone}, which we take to mean that
\begin{multline}
    \label{eq:g-strong-monotone}
    \tag{G-PM}
    \iprod{\subdiff G(x') - \subdiff G(x)}{\inv\TauTilde(x'-x)} \ge 
        \norm{x'-x}_{\invstar\TauTilde\Gamma'}^2 - \psi_{\invstar\TauTilde(\Gamma'-\Gamma)}({x'-x}),
        \\
    \quad (x, x' \in X;\, \TauTilde \in \TauTildeSpace;\, \Gamma' \in [0, \Gamma]+\UTildeSet).
\end{multline}
for some family of functionals $\{\psi_{T}: X \to \R\}$, and a linear operator $0 \le \Gamma \in \linear(X; X)$ which models partial strong monotonicity. The inequality in \eqref{eq:g-strong-monotone}, and all such set inequalities in the remainder of this paper, is understood to hold for all elements of the sets $\subdiff G(x')$  and $\subdiff G(x)$.
The operator $\TauTilde \in \TauTildeSpace$ acts as a \term{testing operator}, and the operator $\Gamma' \in \UTildeSet$ as \term{introduced strong monotonicity}. The functional $\psi_{\invstar\TauTilde(\Gamma'-\Gamma)}$ is a \term{penalty} corresponding to the test and the introduced strong monotonicity.
The role of testing will become more apparent in Section \ref{sec:testing}.

Similarly to \eqref{eq:g-strong-monotone}, we assume that $\subdiff F^*$ is \emph{$(\phi, \TauHatSpace, \UHatSet)$-monotone} with respect to $\TauHatSpace$ in the sense that 
\begin{multline}
    \label{eq:f-monotone}
    \tag{F$^*$-PM}
    \iprod{\subdiff F^*(y') - \subdiff F^*(y) }{\invstar\TauHat(y'-y)} \ge 
    \norm{y'-y}_{\inv\TauHat R}^2 -\phi_{\inv\TauHat R}({y'-y}), 
    \\
    \quad (y, y' \in Y;\, \TauHat \in \TauHatSpace;\, R \in \UHatSet)
\end{multline}
for some family of functionals $\{\phi_{T}: Y \to \R\}$.
Again, the inequality in \eqref{eq:f-monotone} is understood to hold for all elements of the sets $\subdiff F^*(y')$ and $\subdiff F^*(y)$.

In our general analysis, we do not set any conditions on $\psi$ and $\phi$, as their role is simply symbolic transfer of dissatisfaction of strong monotonicity into a penalty in our abstract convergence results.
As discussed in the introduction, these functionals can be seen as an abstract approach to smoothing, however without any ``restarting'' requirements on the algorithm.
With this, $\UTildeSet$ and $\UHatSet$ can be seen as sets of admissible smoothing parameters.

Let us next look at a few examples on how \eqref{eq:g-strong-monotone} or \eqref{eq:f-monotone} might be satisfied. First we have the very well-behaved case of quadratic functions.

\begin{example}
    $G(x)=\norm{f-Ax}^2/2$ satisfies \eqref{eq:g-strong-monotone} with $\Gamma=A^*A$, $\UTildeSet=\{0\}$, and $\psi \equiv 0$ for any invertible $\TauTilde$.
\end{example}

The next lemma demonstrates what can be done when all the parameters are scalar. It naturally extends to functions of the form $G(x_1, x_2)=G(x_1)+G(x_2)$ with corresponding product-form parameters.

\begin{lemma}
    \label{lemma:bounded-strong-convexity-constant}
    Let $G: X \to \extR$ be convex, proper, and lower semicontinuous, with $\Dom G$ bounded.
    Then
    \begin{equation}
        \label{eq:strong-convexity-constant}
        G(x')-G(x) \ge \iprod{\subdiff G(x)}{x'-x} + \frac{\gamma}{2}\left(\norm{x'-x}^2-C_\psi\right),
    \end{equation}
    for some constant $C_\psi \ge 0$, every $\gamma \ge 0$, and $x, x' \in X$.
\end{lemma}

\begin{proof}
    We denote $A \defeq \Dom G$.
    If $x' \not \in A$, we have $G(x')=\infty$ so \eqref{eq:strong-convexity-constant} holds irrespective of $\gamma$ and $C$. If $x \not \in A$, we have $\subdiff G(x)=\emptyset$, so \eqref{eq:strong-convexity-constant} again holds. We may therefore compute the constants based on $x, x' \in A$.
    Now, there is a constant $M$ such that $\sup_{x \in A} \norm{x} \le M$. 
    Then $\norm{x'-x} \le 2M$. Thus, if we pick $C=4M^2$, then $(\gamma/2)(\norm{x'-x}^2-C) \le 0$ for every $\gamma \ge 0$ and $x, x' \in A$.
    By the convexity of $G$, \eqref{eq:strong-convexity-constant} holds.
\end{proof}

\begin{example}
    An indicator function $\delta_A$ of a convex bounded set $A$ satisfies the conditions of Lemma \ref{lemma:bounded-strong-convexity-constant}.
    This is generally what we will use and need.
\end{example}

\begin{Algorithm}
    \caption{Primal-dual algorithm with partial acceleration}
    \label{alg:alg}
    \begin{algorithmic}[1]
    \REQUIRE 
    $F^*$ and $G$ satisfying \eqref{eq:g-strong-monotone} and \eqref{eq:f-monotone} for some sets and spaces $\UTildeSet$, $\UHatSet$, $\TauTildeSpace$, $\TauHatSpace$, and $0 \le \Gamma \in \L(X; X)$.
    Initial invertible $\Tau_0 \in \TauSpace$, $\TauTilde_0 \in \TauTildeSpace$, $\TauHat_1 \in \TauHatSpace$, and $\Sigma_1 \in \L(Y; Y)$, as well as $\delta \in (0, 1)$, satisfying for $j=0$ the condition
    \begin{equation}
        \label{eq:simipos-alg}
        S_{j}M_{j} \ge \delta \begin{pmatrix} \invstar\TauTilde_{j} \inv\Tau_{j} & 0 \\ 0 & 0 \end{pmatrix}.
    \end{equation}

    \STATE Choose initial iterates $x^0 \in X$ and $y^0 \in Y$.
    \REPEAT 
    \STATE Find invertible $\Tau_{i+1} \in \TauSpace$, $\TauTilde_{i+1} \in \TauTildeSpace$, $\TauHat_{i+2} \in \TauHatSpace$, and $\Sigma_{i+2} \in \linear(Y; Y)$ satisfying \eqref{eq:simipos-alg} with $j=i+1$, as well as the condition
    \[
        S_i(M_i+\bar\Gamma_i) \ge S_{i+1}M_{i+1}
    \]
    for some $0 \le R_{i+1} \in \UHatSet$ and $\Gamma_i \in [0, \Gamma] + \UTildeSet$.
    \STATE Perform the updates
    \begin{align}
        \notag
        \nextx & \defeq (I+\Tau_i \subdiff G)^{-1}(\thisx - \Tau_i K^* y^{i}),\\
        %
        \notag
        \overnext{w} & \defeq \TauHat_{i+1} K \inv \TauTilde_i (\nextx-\thisx)+K\nextx,
        \\
        \notag
        \nexty & \defeq (I+\Sigma_{i+1} \subdiff F^*)^{-1}(\thisy + \Sigma_{i+1} \overnext{w}).
    \end{align}   
    \UNTIL a stopping criterion is fulfilled.
    \end{algorithmic}
\end{Algorithm}

\subsection{A general algorithm and the idea of testing}
\label{sec:testing}

The only change we make to the proximal point formulation \eqref{eq:pp} of the method \eqref{eq:cp}, is to replace the basic step length or preconditioning operator $M_{\text{basic},i}$ by the operator
\[
    M_i \defeq
        \begin{pmatrix}
            \inv\Tau_i & -K^* \\
            -\TauHat_{i+1} K \inv \TauTilde_i & \inv\Sigma_{i+1}
        \end{pmatrix}.
\]
As we have remarked, the operators $\TauHat_{i+1}$ and $\TauTilde_i$ play the role of $\omega_i$, acting from both sides of $K$.
Our proposed algorithm can thus be characterised as solving on each iteration $i \in \N$ for the next iterate $\nextu$ the preconditioned proximal point problem
\begin{equation}
    \label{eq:prox-update}
    \tag{PP}
    0 \in H(\nextu) + M_i(\nextu-\thisu).
\end{equation}

To study the convergence properties of \eqref{eq:prox-update},
we define the \emph{testing operator}
\begin{equation}
    \label{eq:si-def}
    S_i \defeq
        \begin{pmatrix}
            \invstar \TauTilde_i & 0 \\
            0 & \inv{\TauHat_{i+1}}
        \end{pmatrix}.
\end{equation}
It will turn out that multiplying or ``testing'' \eqref{eq:prox-update} by this operator will allow us to derive convergence rates. This is roughly akin to how distributions (generalised functions) are applied to smooth test functions, hence the terminology.
The testing of \eqref{eq:prox-update} by $S_i$ is also why we introduced testing into the monotonicity conditions \eqref{eq:g-strong-monotone} and \eqref{eq:f-monotone}.
If we only tested \eqref{eq:prox-update} with $S_i=I$, we could at most obtain ergodic convergence of the duality gap for the unaccelerated method. But by testing with something approriate and faster increasing, such as \eqref{eq:si-def}, we are able to extract better convergence rates from \eqref{eq:prox-update}. 

We also set
\[
    \bar\Gamma_i=
    \begin{pmatrix}
        2\Gamma_i & \TauTilde_i^*(K\inv\TauTilde_i-\inv\TauHat_{i+1}K)^* \\
        \TauHat_{i+1} (K\inv\TauTilde_i - \inv \TauHat_{i+1} K) & 2 R_{i+1}
    \end{pmatrix},
\]
for some $\Gamma_i \in [0, \Gamma]+\UTildeSet$ and $R_{i+1} \in \UHatSet$.
We will see in Section \ref{thm:convergence-result-main} that $\bar\Gamma_i$ is a factor of partial strong monotonicity for $H$ with respect to testing by $S_i$.
With this, taking a fixed $\delta>0$, the properties
\begin{align}
    \label{eq:simigamma-simiplus}
    \tag{C1}
    S_i(M_i+\bar\Gamma_i) & \ge S_{i+1}M_{i+1},
    \quad\text{and}\quad
    \\
    \label{eq:simipos}
    \tag{C2}
    S_iM_i & \ge \delta \begin{pmatrix} \invstar\TauTilde_i \inv\Tau_i & 0 \\ 0 & 0 \end{pmatrix}
    \ge 0,
\end{align}
will turn out to be the crucial defining properties for the convergence rates of the iteration \eqref{eq:prox-update}. The resulting method can also be expressed as Algorithm \ref{alg:alg}. The main steps in developing practical algorithms based on it, will be in the choice of the various step length operators. This will be the content of Section \ref{sec:scalar0} and \ref{sec:special}. Before this, we expand the conditions \eqref{eq:simigamma-simiplus} and \eqref{eq:simipos} to see how they might be satisfied, and study abstract convergence results.

\subsection{A simplified condition}

We expand
\[
    S_iM_i = \begin{pmatrix}
        \invstar\TauTilde_i \inv\Tau_i & - \invstar \TauTilde_i K^* \\
        -K\inv \TauTilde_i & \inv\TauHat_{i+1}\inv\Sigma_{i+1}
    \end{pmatrix},
\]
as well as
\begin{equation}
    \label{eq:si-bargammai}
    S_i\bar\Gamma_i= \begin{pmatrix}
        2 \invstar\TauTilde_i \Gamma_i & \invstar\TauTilde_i K^* - K^* \invstar \TauHat_{i+1}  \\
        K\inv\TauTilde_i-\inv \TauHat_{i+1}K &  2 \inv\TauHat_{i+1} R_{i+1}
    \end{pmatrix},
\end{equation}
and
\[
    S_i(M_i +\bar\Gamma_i)= \begin{pmatrix}
        \invstar\TauTilde_i(\inv\Tau_i +2\Gamma_i) & - K^* \invstar \TauHat_{i+1}  \\
        -\inv \TauHat_{i+1}K & \inv\TauHat_{i+1}(\inv\Sigma_{i+1} + 2 R_{i+1})
    \end{pmatrix}.
\]
By Young's inequality,
\eqref{eq:simipos} is thus satisfied when for some invertible $Z_i \in \linear(X; X)$,
\begin{equation}
    \notag
    \inv\TauHat_{i+1}\inv\Sigma_{i+1} \ge K \inv Z_i \invstar Z_i K^*,
    \quad\text{and}\quad
    (1-\delta)\invstar\TauTilde_i \inv\Tau_i \ge \invstar\TauTilde_i Z_i^* Z_i\inv\TauTilde.
\end{equation}
The second constraint is satisfied as an equality if
\begin{equation}
    \label{eq:zi-choice}
    Z_i^*Z_i=(1-\delta)\inv \Tau_i \TauTilde_i.
\end{equation}
Note that this choice will also be optimal for the first constraint. By the spectral theorem for self-adjoint operators on Hilbert spaces (e.g., \cite[Chapter 12]{rudin2006functional}), we can make the choice \eqref{eq:zi-choice} if
\begin{equation}
    \notag
    \inv \Tau_i\TauTilde_i \in \mathcal{Q} \defeq \{A \in \L(X; X) \mid A \text{ is self-adjoint and positive definite}\}.
\end{equation}
Equivalently, by the same spectral theorem, $\inv\TauTilde_i\Tau_i \in \mathcal{Q}$.
Therefore \eqref{eq:simipos} holds when
\begin{equation}
    \label{eq:tauhatsigma-bound}
    \tag{C2$'$}
    \inv\TauTilde_i\Tau_i \in \mathcal{Q}
    \quad\text{and}\quad
    \inv\TauHat_{i+1}\inv\Sigma_{i+1} \ge \frac{1}{1-\delta} K \inv\TauTilde_i \Tau_i K^*.
\end{equation}
Also, \eqref{eq:simigamma-simiplus} can be rewritten
\begin{equation}
    \label{eq:simigamma-simiplus-2}
    \tag{C1$'$}
    \begin{pmatrix}
        \invstar\TauTilde_i(\inv\Tau_i +2\Gamma_i) - \invstar\TauTilde_{i+1} \inv\Tau_{i+1} &  \invstar \TauTilde_{i+1} K^* - K^*\invstar\TauHat_{i+1} \\
        K\inv\TauTilde_{i+1}-\inv \TauHat_{i+1}K & \inv\TauHat_{i+1}(\inv\Sigma_{i+1} + 2 R_{i+1})- \inv\TauHat_{i+2}\inv\Sigma_{i+2}
    \end{pmatrix}
    \ge 0.
\end{equation}

\subsection{Basic convergence result}
\label{sec:convergence-result-main}

Our main result on Algorithm \ref{alg:alg} is the following theorem, providing some general convergence estimates. It is, however, important to note that the theorem does not yet directly prove convergence, as its estimates depend on the rate of decrease of $\Tau_N \TauTilde_N^*$, as well as the rate of  increase of the penalty sum $\sum_{i=0}^{N-1} D_{i+1}$ coming from the dissatisfaction of strong convexity. Deriving these rates in special cases will be the topic of Section \ref{sec:special}.

\begin{theorem}
    \label{thm:convergence-result-main}
    Let us be given $K \in \linear(X; Y)$, and convex, proper, lower semicontinuous functionals $G: X \to \extR$ and $F^*: Y \to \extR$  on Hilbert spaces $X$ and $Y$, satisfying \eqref{eq:g-strong-monotone} and \eqref{eq:f-monotone}.
    Pick $\delta \in (0, 1)$, and suppose \eqref{eq:simigamma-simiplus} and \eqref{eq:simipos} are satisfied for each $i \in \N$ for some invertible $\Tau_{i} \in \TauSpace$, $\TauTilde_{i} \in \TauTildeSpace$, $\TauHat_{i+1} \in \TauHatSpace$, and $\Sigma_{i+1} \in \linear(Y; Y)$, as well as $\Gamma_i \in [0, \Gamma] + \UTildeSet$ and $R_{i+1} \in \UHatSet$.
    Let $\realoptu=(\realoptx, \realopty)$ satisfy \eqref{eq:oc}.
    Then the iterates of Algorithm \ref{alg:alg} satisfy
    \begin{equation}
        \label{eq:convergence-result-main}
        \frac{\delta}{2}\norm{x^N-\realoptx}_{\invstar \TauTilde_N \inv \Tau_N}^2
        \le
        C_0
        +
        \sum_{i=0}^{N-1} D_{i+1},
        \quad
        (N \ge 1),
    \end{equation}
    for
    \begin{equation}
        \label{eq:diplus1}
        D_{i+1} \defeq
            \psi_{\invstar\TauTilde_i(\Gamma_i-\Gamma)}({x^{i+1}-\realoptx})
            +
            \phi_{\inv\TauHat_{i+1} R_{i+1}}({y^{i+1}-\realopty}),
        \quad\text{and}\quad
        C_0 \defeq \frac{1}{2}\norm{u^0-\realoptu}_{S_0 M_0}^2.
    \end{equation}
\end{theorem}

\begin{remark}
    The term $D_{i+1}$ coming from the dissatisfaction of strong convexity, penalises the basic convergence rate.
    If $\Tau_N \TauTilde_N$ is of the order $O(1/N^2)$, at least on a subspace, and we can bound the penalty $D_{i+1} \le C$ for some constant $C$, then we clearly obtain mixed $O(1/N^2) + O(1/N)$ convergence rates on the subspace.
    If we can assume that $D_{i+1}$ actually converges to zero at some rate, then it will even be possible to obtain improved convergence rates. 
    Since typically $\TauTilde_i, \TauHat_{i+1} \downto 0$ reduce to scalar factors within $D_{i+1}$, this would require prior knowledge of the rates of convergence $x^i \to \realoptx$ and $y^i \to \realopty$. Boundedness we can however usually ensure.
\end{remark}

\begin{proof}
Since $0 \in H(\realoptu)$, we have
\[
    \iprod{H(\nextu)}{S_i^*(\nextu - \realoptu)}
    \subset
    \iprod{H(\nextu)-H(\realoptu)}{S_i^*(\nextu - \realoptu)}.
\]
Recalling the definition of $S_i$ from \eqref{eq:si-def}, and of $H$ from \eqref{eq:h}, it follows
\begin{equation}
    \notag
    \begin{split}
        \iprod{H(\nextu)}{S_i^*(\nextu - \realoptu)}
        &
        \subset \iprod{\subdiff G(\nextx)-\subdiff G(\realoptx)}{\inv \TauTilde_i(\nextx - \realoptx)}
        \\ & \phantom{ = }
         + \iprod{\subdiff F^*(\nexty)-\subdiff F^*(\realopty)}{\invstar{\TauHat_{i+1}}(\nexty - \realopty)}
        \\ & \phantom{ = }
        + \iprod{K^*(\nexty-\realopty)}{\inv \TauTilde_i(\nextx-\realoptx)}
        \\ & \phantom{ = }
        - \iprod{K(\nextx-\realoptx)}{\invstar{\TauHat_{i+1}}(\nexty-\realopty)}.
    \end{split}
\end{equation}
An application of \eqref{eq:g-strong-monotone} and \eqref{eq:f-monotone} consequently gives
\begin{equation}
    \notag
    \begin{split}
        \iprod{H(\nextu)}{S_i^*(\nextu - \realoptu)}
        &
        \ge 
            \norm{\nextx-\realoptx}_{\invstar\TauTilde_i\Gamma_i}^2
            +\norm{\nexty-\realopty}_{\inv\TauHat_{i+i}R_{i+1}}^2
        \\ & \phantom{ = } 
            -\phi_{\inv\TauHat_{i+1}R_{i+1}}({\nexty-\realopty})
            -\psi_{\invstar\TauTilde_i(\Gamma_i-\Gamma)}({\nextx-\realoptx})
        \\ & \phantom{ = }
        + \iprod{K \inv \TauTilde_i(\nextx-\realoptx)}{\nexty-\realopty}
        - \iprod{\inv{\TauHat_{i+1}}K(\nextx-\realoptx)}{\nexty-\realopty}.
    \end{split}
\end{equation}
Using the expression \eqref{eq:si-bargammai} for $S_i\bar\Gamma_i$, and \eqref{eq:diplus1} for $D_{i+1}$, we thus deduce
\begin{equation}
    \label{eq:sc-est}
        \iprod{H(\nextu)}{S_i^*(\nextu - \realoptu)}
        \ge
        \frac{1}{2}\norm{\nextu-\realoptu}_{S_i\bar\Gamma_i}^2
            -D_{i+1}.
\end{equation}

For arbitrary $M \in \L(X \times Y; X \times Y)$ we calculate
\begin{equation}
    \label{eq:M-split}
    \iprod{\nextu-\thisu}{\nextu-\realoptu}_M
    = \frac{1}{2}\norm{\nextu-\thisu}_{M}^2
        - \frac{1}{2}\norm{\thisu-\realoptu}_{M}^2
        + \frac{1}{2}\norm{\nextu-\realoptu}_{M}^2.
\end{equation}
In particular
\begin{equation}
    \label{eq:h-est-ma}
    \notag
        \iprod{M_i(\thisu-\nextu)}{S_i^*(\nextu - \realoptu)}
        =
        -\frac{1}{2}\norm{\nextu-\thisu}_{S_i  M_i}^2
        + \frac{1}{2}\norm{\thisu-\realoptu}_{S_i M_i}^2
        - \frac{1}{2}\norm{\nextu-\realoptu}_{S_i M_i}^2.
\end{equation}
Using \eqref{eq:simigamma-simiplus} to estimate $\frac{1}{2}\norm{\nextu-\realoptu}_{S_i M_i}^2$ and \eqref{eq:simipos} to eliminate $\frac{1}{2}\norm{\nextu-\thisu}_{S_i  M_i}^2$ yields
\begin{equation}
    \label{eq:h-est-ma-x}
    \iprod{M_i(\thisu-\nextu)}{S_i^*(\nextu - \realoptu)}
    \le 
        \frac{1}{2}\norm{\thisu-\realoptu}_{S_i M_i}^2
        - \frac{1}{2}\norm{\nextu-\realoptu}_{S_{i+1} M_{i+1}}^2
        + \frac{1}{2}\norm{\nextu-\realoptu}_{S_i\bar\Gamma_i}^2.
\end{equation}
Combining \eqref{eq:sc-est} and \eqref{eq:h-est-ma-x} through \eqref{eq:prox-update}, we thus obtain
\begin{equation}
    \label{eq:desc-est-1}
    \frac{1}{2}\norm{\nextu-\realoptu}_{S_{i+1} M_{i+1}}^2
    \le
    \frac{1}{2}\norm{\thisu-\realoptu}_{S_i M_i}^2
    +D_{i+1}.
\end{equation}
Summing \eqref{eq:desc-est-1} over $i=1,\ldots,N-1$, and applying \eqref{eq:simipos} to estimate
\[
     \frac{\delta}{2}\norm{\nextx-\realoptx}_{\invstar\TauTilde_N\inv\Tau_N} \le \frac{1}{2}\norm{\nextu-\realoptu}_{S_{N} M_{N}}^2,
\]
we obtain \eqref{eq:convergence-result-main}.
\end{proof}

\section{Scalar diagonal updates and the ergodic duality gap}
\label{sec:scalar0}

One relatively easy way to satisfy \eqref{eq:g-strong-monotone}, \eqref{eq:f-monotone}, \eqref{eq:simigamma-simiplus} and \eqref{eq:simipos}, is to take the ``diagonal'' step length operators $\TauHat_i$ and $\TauTilde_i$ as equal scalars. Another good starting point would be to choose $\TauTilde_i=\Tau_i$. We however do not explore this route in the present work, instead specialising now Theorem \ref{thm:convergence-result-main} to the scalar case. We then explore ways to add estimates of the ergodic duality gap into \eqref{eq:convergence-result-main}. While this would be possible in the general framework through convexity notions analogous to \eqref{eq:g-strong-monotone} and \eqref{eq:f-monotone}, the resulting gap would not be particularly meaningful.
We therefore concentrate on the scalar diagonal updates to derive estimates on the ergodic duality gap.

\begin{Algorithm}
    \caption{Primal-dual algorithm with partial acceleration---partially scalar}
    \label{alg:alg-scalar}
    \begin{subequations}
    \label{eq:alg-scalar-updates}
    \begin{algorithmic}[1]
    \REQUIRE 
    $F^*$ and $G$ satisfying \eqref{eq:g-strong-monotone-scalar} and \eqref{eq:f-monotone-scalar}  for some sets $\UTildeSet$, $\UHatSet$, and $0 \le \Gamma \in \L(X; X)$.
    A choice of $\delta \in (0, 1)$.
    Initial invertible step length operators $\Tau_0 \in \mathcal{Q}$ and $\Sigma_0 \in \L(Y; Y)$, as well as step length parameter $\tilde\tau_0 > 0$.
    \STATE Choose initial iterates $x^0 \in X$ and $y^0 \in Y$.
    \REPEAT
    \STATE Find $\tilde\omega_i>0$,  $\OverRelax_i \in \linear(X; X)$, and $\Gamma_i \in [0, \Gamma] + \UTildeSet$ satisfying
    \begin{gather}
        \label{eq:alg-scalar-paramupd}
        \tilde\omega_i(I+2\Gamma_i \Tau_i)\OverRelax_i 
         \ge I,
         \quad\text{and}\quad
         \Tau_i\OverRelax_i \in \mathcal{Q}.
    \end{gather}
    \STATE Set
    \begin{equation}
        \label{eq:alg-scalar-tauupd}
        \Tau_{i+1} \defeq \Tau_i\OverRelax_i,
        \quad\text{and}\quad
        \tilde\tau_{i+1} \defeq \tilde\tau_i\tilde\omega_i.
    \end{equation}
    \STATE Find $\Sigma_{i+1} \in \linear(Y; Y)$ and $R_{i} \in \UHatSet$ satisfying
    \begin{align}
        \label{eq:alg-scalar-sigmaupd-2}
        \inv\Sigma_{i} + 2 R_{i} & \ge \inv{\tilde\omega_{i}}\inv\Sigma_{i+1} \ge  \inv{(1-\delta)}  K \Tau_i K^*.
    \end{align}
    \STATE Perform the updates
    \label{eq:alg-scalar}
    \begin{align}
        \notag
        \nextx & \defeq (I+\Tau_i \subdiff G)^{-1}(\thisx - \Tau_i K^* y^{i}),\\
        \notag
        \overnextx & \defeq \tilde\omega_i (\nextx-\thisx)+\nextx,
        \\
        \notag
        \nexty & \defeq (I+\Sigma_{i+1} \subdiff F^*)^{-1}(\thisy + \Sigma_{i+1} K \overnextx).
    \end{align}   
    \UNTIL a stopping criterion is fulfilled.
    \end{algorithmic}
    \end{subequations}
\end{Algorithm}

\subsection{Scalar specialisation of Algorithm \ref{alg:alg}}
\label{sec:scalar}

We take $\tilde\OverRelax_i=\tilde\omega_i I$, $\TauTilde_i=\tilde\tau_i I$, and $\TauHat_i=\tilde\tau_i I$ for some $\tilde\omega_i, \tilde\tau_i > 0$. With this \eqref{eq:tauhatsigma-bound} becomes
\begin{equation}
    \label{eq:tauhatsigma-bound-scalar}
    \tag{C2$''$}
    \Tau_i \in \mathcal{Q},
    \quad\text{and}\quad
    \inv\Sigma_{i+1} \ge \tilde\omega_i \inv{(1-\delta)} K \Tau_i K^*,
\end{equation}
while, the diagonal terms cancelling out, \eqref{eq:simigamma-simiplus-2} becomes
\begin{equation}
    \label{eq:simigamma-simiplus-2-scalar}
    \tag{C1$''$}
    \begin{aligned}
        \inv{\tilde\tau}_i(I +2\Gamma_i \Tau_i)\inv\Tau_i & \ge \inv{\tilde\tau}_{i+1} \inv\Tau_{i+1},
        \quad\text{and} \\
        \inv{\tilde\tau}_{i+1}(\inv\Sigma_{i+1} + 2 R_{i+1}) & \ge \inv{\tilde\tau}_{i+2}\inv\Sigma_{i+2}.
    \end{aligned}
\end{equation}

For simplicity, we now assume $\phi$ and $\psi$ to satisfy the identities
\begin{equation}
    \label{eq:phi-assumptions}
    \psi_T(-x)=\psi_T(x),
    \quad\text{and}\quad
    \psi_{\alpha T}(x)=\alpha\psi_T(x),
    \quad
    (x \in X;\, 0 < \alpha \in \R).
\end{equation}
The monotonicity conditions \eqref{eq:g-strong-monotone} and \eqref{eq:f-monotone} then simplify into
\begin{gather}
    \label{eq:g-strong-monotone-scalar}
    \tag{G-pm}
    \iprod{\subdiff G(x') - \subdiff G(x)}{x'-x} \ge 
        \norm{x'-x}_{\Gamma'}^2-\psi_{\Gamma'-\Gamma}({x'-x}),
    \quad (x, x' \in X;\, \Gamma' \in [0, \Gamma] + \UTildeSet),
\intertext{and}
    \label{eq:f-monotone-scalar}
    \tag{F$^*$-pm}
    \iprod{\subdiff F^*(y') - \subdiff F^*(y) }{y'-y} \ge 
    \norm{y'-y}_{R}^2-\phi_{R}({y'-y}),
    \quad (y, y' \in Y;\, R \in \UHatSet).
\end{gather}

We have thus converted the main conditions \eqref{eq:simipos}, \eqref{eq:simigamma-simiplus}, \eqref{eq:g-strong-monotone}, and \eqref{eq:f-monotone} of Theorem \ref{thm:convergence-result-main} into the respective conditions \eqref{eq:tauhatsigma-bound-scalar}, \eqref{eq:simigamma-simiplus-2-scalar}, \eqref{eq:g-strong-monotone-scalar}, and \eqref{eq:f-monotone-scalar}.
Rewriting \eqref{eq:simigamma-simiplus-2-scalar} in terms of $\Omega_i$ and $\tilde\omega_i$ satisfying $\Tau_{i+1}=\Tau_i\Omega_i$ and $\tilde\tau_{i+1}=\tilde\tau_i\tilde\omega_i$, we reorganise \eqref{eq:simigamma-simiplus-2-scalar} and \eqref{eq:tauhatsigma-bound-scalar} into the parameter update rules \eqref{eq:alg-scalar-updates} of Algorithm \ref{alg:alg-scalar}.
For ease of expression, we introduce there $\Sigma_0$ and $R_0$ as dummy variables that are not used anywhere else.
Equating $\overnext{w}=K\overnext{x}$, we observe that  Algorithm \ref{alg:alg-scalar} is an instance of Algorithm \ref{alg:alg}.
Observe that $\tilde\tau_i$ and $\hat\tau_i$ disappear from the algorithm aside from the residual factor $\tilde\omega_i$, which can give different over-relaxation rates in the rule for $\bar x^{i+1}$ compared to $\omega_i$ in \eqref{eq:cp}.
Moreover, the parameter $\tilde\tau_i$ will still play a critical role in our study of convergence rate estimates.

\begin{example}[The method of Chambolle and Pock]
    \label{ex:chambolle-pock}
    Let $G$ be strongly convex with factor $\gamma \ge 0$. We take $\Tau_i=\tau_i I$, $\TauTilde_i=\tau_i I$, $\TauHat_i=\tau_i I$, and $\Sigma_{i+1}=\sigma_{i+1} I$ for some scalars $\tau_i, \sigma_{i+1}>0$.
    The conditions \eqref{eq:g-strong-monotone-scalar} and \eqref{eq:f-monotone-scalar} then hold with $\psi \equiv 0$ and $\phi \equiv 0$, while \eqref{eq:tauhatsigma-bound-scalar} and \eqref{eq:simigamma-simiplus-2-scalar} reduce with $R_{i+1}=0$, $\Gamma_i=\gamma I$, $\Omega_i=\omega_i I$, and $\tilde\omega_i=\omega_i$ into 
    \begin{equation}
        \notag
        \omega_i^2(1+2\gamma\tau_i) \ge 1, \quad\text{and}\quad
        (1-\delta)/\norm{K}^2 \ge  \tau_{i+2}\sigma_{i+2} \ge \tau_{i+1}\sigma_{i+1}.
    \end{equation}
    Updating $\sigma_{i+1}$ such that the last inequality holds as an equality, we recover the accelerated PDHGM \eqref{eq:cp}+\eqref{eq:cpaccel}. If $\gamma=0$, we recover the unaccelerated PDHGM.
\end{example}

\subsection{The ergodic duality gap and convergence}

To study the convergence of an ergodic duality gap, we now introduce convexity notions analogous to \eqref{eq:g-strong-monotone-scalar} and \eqref{eq:f-monotone-scalar}. Namely, we assume
\begin{multline}
    \label{eq:g-strong-convexity-scalar}
    \tag{G-pc}
    G(x')-G(x) \ge \iprod{\subdiff G(x)}{x'-x} + 
    \frac{1}{2}\norm{x'-x}_{\Gamma'}^2 - \frac{1}{2}\psi_{\Gamma'-\Gamma}({x'-x}),
    \\
    \quad (x, x' \in X;\, \Gamma' \in [0, \Gamma]+\UTildeSet).
\end{multline}
and
\begin{multline}
    \label{eq:f-strong-convexity-scalar}
    \tag{F$^*$-pc}
    F^*(y')-F^*(y) \ge \iprod{\subdiff  F^*(y)}{y'-y} 
    +
    \frac{1}{2}\norm{y'-y}_{R}^2 - \frac{1}{2}\psi_{R}({y'-y}),
    \\
    \quad (y, y' \in Y;\, R \in \UHatSet).
\end{multline}
It is easy to see that these imply \eqref{eq:g-strong-monotone-scalar} and \eqref{eq:f-monotone-scalar}. 

To define an ergodic duality gap, we set
\begin{equation}
    \label{eq:qn}
    \tilde q_N \defeq \sum_{j=0}^{N-1} \inv{\tilde\tau_j},
    \quad\text{and}\quad
    \hat q_N \defeq \sum_{j=0}^{N-1} \inv{\hat\tau_{j+1}},
\end{equation}
and define the weighted averages
\begin{equation}
    \notag
    x_N \defeq  \inv{\tilde q_N} \sum_{i=0}^{N-1} \inv {\tilde\tau_i} \nextx,
    \quad\text{and}\quad
    y_N \defeq \inv{\hat q_N} \sum_{i=0}^{N-1}  \inv {\hat\tau_{i+1}} \nexty.
\end{equation}
With these, the ergodic duality gap at iteration $N$ is defined as
\begin{equation}
    \notag
    \ergGap{N}
    \defeq            
        \bigl(G(x_N) + \iprod{\realopty}{Kx_N}  - F^*(\realopty)\bigr)
        -\bigl(G(\realoptx) + \iprod{y_N}{K \realoptx} - F^*(y_N)\bigr),
\end{equation}
and we have the following convergence result.

\begin{theorem}
    \label{thm:convergence-result-gap}
    Let us be given $K \in \linear(X; Y)$, and convex, proper, lower semicontinuous functionals $G: X \to \extR$ and $F^*: Y \to \extR$ on Hilbert spaces $X$ and $Y$, satisfying \eqref{eq:g-strong-convexity-scalar} and \eqref{eq:f-strong-convexity-scalar} for some sets $\UTildeSet$, $\UHatSet$, and $0 \le \Gamma \in \L(X; X)$.
    Pick $\delta \in (0, 1)$, and suppose \eqref{eq:tauhatsigma-bound-scalar} and \eqref{eq:simigamma-simiplus-2-scalar} are satisfied for each $i \in \N$ for some invertible $\Tau_i \in \mathcal{Q}$, $\Sigma_{i} \in \L(Y; Y)$, 
    \begin{equation}
        \label{eq:tilde-tau-decreasing}
        \tag{C3$''$}
        0 < \tilde\tau_{i} \le \tilde\tau_0,
    \end{equation}
    as well as $\Gamma_i \in ([0, \Gamma] + \UTildeSet)/2$ and $R_{i} \in \UHatSet/2$.
    Let $\realoptu=(\realoptx, \realopty)$ satisfy \eqref{eq:oc}.    
    Then the iterates of Algorithm \ref{alg:alg-scalar} satisfy
    \begin{equation}
        \label{eq:convergence-result-gap-scalar}
        \frac{\delta}{2}\norm{x^N-\realoptx}_{\inv{\tilde\tau_N} \inv \Tau_N}^2
        + \tilde q_N\ergGap{N}
        \le
        C_0
        +
        \sum_{i=0}^{N-1} D_{i+1}.
    \end{equation}
    Here $C_0$ is as in \eqref{eq:diplus1}, and $D_{i+1}$ simplifies into
    \begin{equation}
        \label{eq:diplus1-scalar}
        D_{i+1} =
            \inv{\tilde\tau_i}\psi_{(\Gamma_i-\Gamma)}({x^{i+1}-\realoptx})
            +
            \inv{\tilde\tau_{i+1}}\phi_{R_{i+1}}({y^{i+1}-\realopty}).
    \end{equation}

    If only \eqref{eq:g-strong-monotone-scalar} and \eqref{eq:f-monotone-scalar} hold instead of \eqref{eq:g-strong-convexity-scalar} and \eqref{eq:f-strong-convexity-scalar}, or we take $R_{i} \in \UHatSet$ and $\Gamma_i \in [0,\Gamma]+\UTildeSet$, then \eqref{eq:convergence-result-gap-scalar} holds with $\ergGap{N} \defeq 0$.
\end{theorem}

\begin{remark}
    For convergence of the gap, we must accelerate less (factor $1/2$ on $\Gamma_i$).
\end{remark}

\begin{example}[No acceleration]
    Consider Example \ref{ex:chambolle-pock}, where $\psi \equiv 0$ and $\phi \equiv 0$. If $\gamma=0$, we get ergodic convergence of the duality gap at rate $O(1/N)$. Indeed, we are in the scalar step setting, with $\hat\tau_j=\tilde\tau_j=\tau_0$. Thus presently $\tilde q_N=N\tau_0$.
\end{example}

\begin{example}[Full acceleration]
    With $\gamma>0$ in Example \ref{ex:chambolle-pock}, we know from \cite[Corollary 1]{chambolle2010first} that
    \begin{equation}
         \label{eq:taun-standard-limit}
         \lim_{N \to\infty} N\tau_N\gamma=1.
    \end{equation}
    Thus $\tilde q_N$ is of the order $N^2$. So is $\tilde \tau_N \Tau_N = \tau_N^2 I$. Therefore, \eqref{eq:convergence-result-gap-scalar} shows $O(1/N^2)$ convergence of the squared distance to solution. For $O(1/N^2)$ convergence of the ergodic duality gap, we need to slow down \eqref{eq:cpaccel} to $\omega_i = 1/\sqrt{1+\gamma\tau_i}$.
\end{example}

\begin{remark}
    \label{rem:improved-tau-estimate}
    The result \eqref{eq:taun-standard-limit} can be improved to estimate $\tau_N \le C_\tau/N$ without a qualifier $N \ge N_0$.
    Indeed, from \cite[Lemma 2]{chambolle2010first} we know for the rule $\omega_i=1/\sqrt{1+2\gamma\tau_i}$ that given $\lambda>0$ and $N \ge 0$ with $\gamma\tau_N \le \lambda$, for any $\ell \ge 0$ holds
    \[
        \frac{1}{\gamma\tau_N} + \frac{\ell}{1+\lambda} 
        \le \frac{1}{\gamma\tau_{N+\ell}} 
        \le \frac{1}{\gamma\tau_N} + \ell.
    \]
    If we pick $N=0$ and $\lambda=\gamma\tau_0$, this says
    \[
        \frac{1}{\gamma\tau_0} + \frac{\ell}{1+\gamma\tau_0} 
        \le \frac{1}{\gamma\tau_{\ell}} 
        \le \frac{1}{\gamma\tau_0} + \ell.
    \]
    In particular,
    \[
        \tau_\ell \le \frac{1}{\gamma\left(\frac{1}{\gamma\tau_0} + \frac{\ell}{1+\gamma\tau_0} \right)}
        =
        \frac{1+\gamma\tau_0}{\inv\tau_0+\gamma\ell}
        \le\frac{\inv\gamma+\tau_0}{\ell}.
    \]
    Therefore, $\tau_N \le C_\tau/N$ for $C_\tau \defeq \inv\gamma+\tau_0$.
    Moreover, $\inv\tau_N \le \inv\tau_0 + \gamma N$.
\end{remark}

\begin{proof}[Proof of Theorem \ref{thm:convergence-result-gap}]
The final non-gap estimate is a direct consequence of Theorem \ref{thm:convergence-result-main}, so we concentrate on the gap estimate.
We begin by expanding
\begin{equation}
    \notag
    \begin{split}
        \iprod{H(\nextu)}{S_i^*(\nextu - \realoptu)}
        &
        = \inv {\tilde\tau_i} \iprod{\subdiff G(\nextx)}{\nextx - \realoptx}
         + \inv{\hat\tau_{i+1}} \iprod{\subdiff F^*(\nexty)}{\nexty - \realopty}
        \\  & \phantom{ = }
        + \inv{\tilde\tau_i}\iprod{K^*\nexty}{\nextx-\realoptx}
        - \inv{\hat\tau_{i+1}} \iprod{K \nextx}{\nexty-\realopty}
    \end{split}
\end{equation}
Since then $\Gamma_i \in ([0, \Gamma] + \UTildeSet)/2$, and $R_{i+1} \in \UHatSet/2$, we may take $\Gamma'=2\Gamma_i$ and $R=2R_{i+1}$ in \eqref{eq:g-strong-convexity-scalar} and \eqref{eq:f-strong-convexity-scalar}. It follows
\begin{equation}
    \notag
    \begin{split}
        \iprod{H(\nextu)}{S_i^*(\nextu - \realoptu)}
        &
        \ge \inv{\tilde\tau_i} 
        \left(
            G(\nextx)- G(\realoptx)
            + \frac{1}{2}\norm{\nextx-\realoptx}_{2\Gamma_i}^2
            - \frac{1}{2}\psi_{2\Gamma_i}({\nextx-\realoptx})
        \right)
        \\  & \phantom{ = }
         + \inv{\hat\tau_{i+1}}
         \left(
            F^*(\nexty)- F^*(\realopty)
            + \frac{1}{2}\norm{\nexty-\realopty}_{2R_{i+1}}^2
            - \frac{1}{2}\phi_{2R_{i+1}}({\nexty-\realopty})
        \right)
        \\  & \phantom{ = }
        - \inv{\tilde\tau_i}\iprod{\nexty}{K\realoptx}
        + \inv{\hat\tau_{i+1}}\iprod{\realopty}{K \nextx}
        + (\inv{\tilde\tau_i} - \inv{\hat\tau_{i+1}})\iprod{\nexty}{K \nextx}.
    \end{split}
\end{equation}
Using \eqref{eq:iprod-def} and \eqref{eq:phi-assumptions}, we can make all of the factors ``2'' and ``1/2'' in this expression annihilate each other. With $D_{i+1}$ as in \eqref{eq:diplus1-scalar} (equivalently \eqref{eq:diplus1}) we therefore have
\begin{equation}
    \notag
    \begin{split}
        \iprod{H(\nextu)}{S_i^*(\nextu - \realoptu)}
        &
        \ge
        \inv{\tilde\tau_i} \left( G(\nextx)-G(\realoptx)+\iprod{\realopty}{K\nextx}\right)
        + \norm{\nextx-\realoptx}_{\inv{\tilde\tau_i} \Gamma_i}^2
        \\ & \phantom { = }        
        + \inv{\hat\tau_{i+1}}\left( F^*(\nexty)-F^*(\realopty)-\iprod{\nexty}{K\realoptx}\right)
        + \norm{\nexty-\realopty}_{\inv{\hat\tau_{i+1}}R_{i+1}}^2
        \\ & \phantom{ = }
        + (\inv{\tilde\tau_i} - \inv{\hat\tau_{i+1}})\left(\iprod{\nexty-\realopty}{K(\nextx-\realoptx)}
        - 
        \iprod{\realopty}{K\realoptx}
        \right)
        - D_{i+1}.
    \end{split}
\end{equation}
A little bit of reorganisation and referral to \eqref{eq:si-bargammai} for the expansion of $S_i\bar\Gamma_i$ thus gives
\begin{equation}
    \label{eq:gapest-x}
    \begin{split}
        \iprod{H(\nextu)}{S_i^*(\nextu - \realoptu)}    
        &
        \ge
        \inv{\tilde\tau_i} \left( G(\nextx)-G(\realoptx) + \iprod{\realopty}{K\nextx}\right)
        \\ & \phantom{ = }
         + \inv{\hat\tau_{i+1}}\left( F^*(\nexty)-F^*(\realopty) - \iprod{\nexty}{K\realoptx}\right)
        \\ & \phantom{ = }
        - (\inv{\tilde\tau_i} - \inv{\hat\tau_{i+1}})\iprod{\realopty}{K\realoptx}
        + \frac{1}{2}\norm{\nextu-\realoptu}_{S_i \bar\Gamma_i}^2 
        - D_{i+1}.
    \end{split}
\end{equation}

Let us write 
\[
    \begin{split}
    \gap^i_+(\nextu, \realoptu) & \defeq
     \bigl(\inv{\tilde\tau_i}G(\nextx)+ \inv{\tilde\tau_i}\iprod{\realopty}{K\nextx}-\inv{\hat\tau_{i}}F^*(\realopty)\bigr)
     \\ & \phantom{ = }
     -\bigl(\inv{\tilde\tau_{i+1}}G(\realoptx)
         + \inv{\hat\tau_{i+1}}\iprod{\nexty}{K\realoptx}
         - \inv{\hat\tau_{i+1}}F^*(\nexty)\bigr).
    \end{split}
\]
Observing here the switches between the indices ${i+1}$ and $i$ of the step length parameters in comparison to the last step of \eqref{eq:gapest-x}, we thus obtain
\begin{equation}
    \label{eq:gap-sc-est}
        \iprod{H(\nextu)}{S_i(\nextu - \realoptu)}
        \ge \gap^i_+(\nextu, \realoptu) - \gap^i_+(\realoptu,\realoptu)
        +
        \frac{1}{2}\norm{\nextu-\realoptu}_{S_i\bar\Gamma_i}^2
        - D_{i+1}.
\end{equation}
Using \eqref{eq:h-est-ma} and \eqref{eq:prox-update}, we now obtain
\begin{equation}
    \notag
    \frac{1}{2}\norm{\nextu - \realoptu}_{S_{i+1} M_{i+1}}^2
    + \gap^i_+(\nextu, \realoptu) - \gap^i_+(\realoptu, \realoptu)
    \le
    \frac{1}{2}\norm{\thisu - \realoptu}_{S_{i} M_{i}}^2
    + D_{i+1}.
\end{equation}
Summing this for $i=0,\ldots,N-1$ gives with $C_0$ from \eqref{eq:diplus1} the estimate
\begin{equation}
    \label{eq:gap-est-semifinal}
        \frac{1}{2}\norm{u^N-\realoptu}_{S_{N} M_{N}}^2
        + \sum_{i=0}^{N-1} \left(\gap^i_+(\nextu, \realoptu) -\gap^i_+(\realoptu, \realoptu)\right)
        \le
        C_0
        +\sum_{i=0}^{N-1} D_{i+1}.
\end{equation}

We want to estimate the sum of the gaps $\gap^i_+$ in \eqref{eq:gap-est-semifinal}.
Using the convexity of $G$ and $F^*$, we observe
\begin{equation}
    \label{eq:gapest0}
    \sum_{i=0}^{N-1} \inv{\tilde\tau_{i}} G(\nextx)
    \ge
    {\tilde q_N} G(x_N),
    \quad\text{and}\quad
    \sum_{i=0}^{N-1} \inv{\hat\tau_{i+1}} F^*(\nexty)
    \ge
    {\hat q_N} F^*(y_N).
\end{equation}
Also, by \eqref{eq:qn} and simple reorganisation
\begin{align}
    \label{eq:gapest1}
    \sum_{i=0}^{N-1} \inv{\tilde\tau_{i+1}}G(\realoptx)
    &
    =
    {\tilde q_N} G(\realoptx)
    +
    \inv{\tilde\tau_{N}} G(\realoptx) - \inv{\tilde\tau_0} G(\realoptx),
    \quad\text{and}
    \\
    \label{eq:gapest3}
    \sum_{i=0}^{N-1} \inv{\hat\tau_{i}} F^*(\realopty)
    &
    =
    {\hat q_N} F^*(y_N)
    -\inv{\hat\tau_N} F^*(\realopty)
    +\inv{\hat\tau_0} F^*(\realopty).
\end{align}
All of \eqref{eq:gapest0}--\eqref{eq:gapest3} together give
\[
    \begin{split}
    \sum_{i=0}^{N-1}
    \gap^i_+(\nextu,\realoptu)
    &
    \ge
    \bigr({\tilde q_N}G(x_N) + \tilde q_N\iprod{\realopty}{Kx_N} - {\hat q_N}F^*(\realopty) \bigl) 
    \\ & \phantom{\ge}
    - \bigl({\tilde q_N}G(\realoptx) + \hat q_N\iprod{y_N}{K\realoptx} - \hat q_NF^*(y_N)\bigr)
    \\ & \phantom{\ge}
    + \left(
        \inv{\tilde\tau_N}G(\realoptx) - \inv{\tilde\tau_0}G(\realoptx)
        + \inv{\hat\tau_N}F^*_{\invstar\TauHat_{N}}(\realoptx) - \inv{\hat\tau_0}F^*(\realopty)
    \right).
    \end{split}
\]
Another use of \eqref{eq:qn} gives
\[
    \sum_{i=0}^{N-1}
    \gap^i_+(\realoptu,\realoptu)
    =
    (\tilde q_N - \hat q_N) \iprod{\realopty}{K\realoptx}
    + \left(
        \inv{\tilde\tau_N} G(\realoptx) - \inv{\tilde\tau_0} G(\realoptx)
        + \inv{\hat\tau_N} F^*(\realoptx) - \inv{\hat\tau_0} F^*(\realopty)
    \right).
\]
Thus
\begin{equation}
    \label{eq:gapsumest}
    \sum_{i=0}^{N-1} \bigl(\gap^i_+(\nextu,\realoptu)-\gap^i_+(\realoptu,\realoptu)\bigr)
    \ge 
    \tilde q_N \ergGap{N} + r_N,
\end{equation}
where the remainder
\begin{equation}
    \notag
    r_N
    =(\tilde q_N-\hat q_N)\left(
        F^*(\hat y)-F^*(y_N) - \iprod{\hat y-y_N}{K \hat x}
        \right).
\end{equation}

At a solution $\realoptu=(\realoptx, \realopty)$ to \eqref{eq:oc}, $K \hat x \in \subdiff F^*(\hat y)$, so $r_N \ge 0$ provided $\tilde q_N \le \hat q_N$. But $\tilde q_N-\hat q_N=\inv {\tilde\tau}_0 - \inv {\tilde\tau}_{N}$, so this is guaranteed by our assumption \eqref{eq:tilde-tau-decreasing}.
Using \eqref{eq:gapsumest} in \eqref{eq:gap-est-semifinal} therefore gives
\begin{equation}
    \label{eq:gap-est-final}
        \frac{1}{2}\norm{u^N-\realoptu}_{S_{N} M_{N}}^2
        + \tilde q_N \ergGap{N} + r_N
        \le
        C_0
        +
        \sum_{i=0}^{N-1} D_{i+1}.
\end{equation}
A referral to \eqref{eq:simipos} to estimate $S_NM_N$ from below shows \eqref{eq:convergence-result-gap-scalar}, concluding the proof.
\end{proof}

\begin{remark}
    We only used the assumption $\tilde\tau_i=\hat\tau_i$ to bound $r_N \ge 0$.
    It is possible to streamline the proof if in addition to this we assume $\{\tilde\tau_i\}$ to be non-increasing instead of merely satisfying \eqref{eq:tilde-tau-decreasing}, and define $y^N$ based on $\tilde q^N$ instead of $\hat q^N$.
\end{remark}
\section{Convergence rates in special cases}
\label{sec:special}

To derive a practical algorithm, we need to satisfy the update rules \eqref{eq:simigamma-simiplus} and \eqref{eq:simipos}, as well as the partial monotonicity conditions \eqref{eq:g-strong-monotone} and \eqref{eq:f-monotone}. As we have already discussed in Section \ref{sec:scalar0}, this is easiest when for some $\tilde\tau_i>0$ we set
\begin{equation}
    \label{eq:tau-scalar-assumption}
    \TauTilde_i=\tilde\tau_i I,
    \quad\text{and}\quad
    \TauHat_i=\tilde\tau_i I.
\end{equation}
The result is Algorithm \ref{alg:alg-scalar}, whose convergence we studied in Theorem \ref{thm:convergence-result-gap}. Our task now is to verify its conditions, in particular \eqref{eq:g-strong-convexity-scalar} and \eqref{eq:f-strong-convexity-scalar} (alternatively \eqref{eq:f-monotone-scalar} and \eqref{eq:g-strong-monotone-scalar}), as well as \eqref{eq:simigamma-simiplus-2-scalar}, \eqref{eq:tauhatsigma-bound-scalar}, and \eqref{eq:tilde-tau-decreasing} for $\Gamma$ of the projection form $\gamma P$.

\subsection{An approach to updating $\Sigma$}
\label{sec:diagonal}

We have not yet defined an explicit update rule for $\Sigma$, merely requiring that it has to satisfy \eqref{eq:tauhatsigma-bound-scalar} and \eqref{eq:simigamma-simiplus-2-scalar}. The former in particular requires
\begin{equation}
    \notag
    \inv\Sigma_{i+1} \ge \tilde\omega_i \inv{(1-\delta)} K \Tau_i K^*.
\end{equation}
Hiring the help of some linear operator $\mathcal{F}\in \linear(\linear(Y; Y); \linear(Y;Y))$ satisfying
\begin{equation}
    \label{eq:ff-lower-bound}
    \mathcal{F}(K \Tau_i K^*) \ge K \Tau_i K^*,
\end{equation}
our approach is to define
\begin{equation}
    \label{eq:tauhatsigma-bound-special-def0}
    \inv\Sigma_{i+1} \defeq {\tilde\omega_i}\inv{(1-\delta)} \mathcal{F}(K \Tau_i K^*).
\end{equation}
Then \eqref{eq:tauhatsigma-bound-scalar} is satisfied provided $\inv \Tau_i \in \mathcal{Q}$. Since
$
    \inv{\tilde\tau_{i+1}}\inv{\Sigma_{i+1}}
    =\inv{\tilde\tau_i}\inv{(1-\delta)} \mathcal{F}(K \Tau_i K^*),
$
the condition \eqref{eq:simigamma-simiplus-2-scalar} reduces into the satisfaction for each $i \in \N$ of
\begin{subequations}
\label{eq:simigamma-simiplus-2-special-simple}
\begin{align}
    \label{eq:simigamma-simiplus-2-special-simple-1}
    \inv{\tilde\tau_i} (I + 2 \Gamma\Tau_i) \inv\Tau_i - \inv{\tilde\tau_{i+1}} \inv\Tau_{i+1}
    & \ge -2\inv{\tilde\tau_i}(\Gamma_i-\Gamma), 
    \quad\text{and}
    \\
    \label{eq:simigamma-simiplus-2-special-simple-2}
    \frac{1}{1-\delta}\left( \inv{\tilde\tau_{i}}\mathcal{F}\left(K\Tau_{i}K^*\right) - \inv{\tilde\tau_{i+1}}\mathcal{F}\left(K\Tau_{i+1} K^*\right)\right)
    & \ge -2 \inv{\tilde\tau_{i+1}} R_{i+1}.
\end{align}
\end{subequations}
To apply Theorem \ref{thm:convergence-result-gap}, all that remains is to verify in special cases these conditions together with \eqref{eq:tilde-tau-decreasing} and the partial strong convexity conditions \eqref{eq:g-strong-convexity-scalar} and \eqref{eq:f-strong-convexity-scalar}.

\subsection{When $\Gamma$ is a projection}
\label{sec:projective}

We now take $\Gamma=\bar \gamma P$ for some $\bar\gamma>0$, and a projection operator $P \in \L(X; X)$: idempotent, $P^2=P$, and self-adjoint, $P^*=P$. We let $P^\perp \defeq I-P$. Then $P^\perp P=P P^\perp=0$. With this, we assume for some $\bar\gamma^\perp>0$ that
\begin{equation}
    \label{eq:bargamma}
    [0, \bar\gamma^\perp P^\perp] \subset \UTildeSet.
\end{equation}

To unify our analysis for gap and non-gap estimates of Theorem \ref{thm:convergence-result-gap}, we now pick $\lambda=1/2$ in the former case, and $\lambda=1$ in the latter. 
We then pick $0 \le \gamma \le \lambda\bar\gamma$, and $0 \le \gamma_i^\perp \le \lambda\bar\gamma^\perp$, and set
\begin{equation}
    \label{eq:projection-tau}
    \Tau_i=\tau_i P+\tau_i^\perp P^\perp,
    \quad
    \OverRelax_i=\omega_i P+\omega_i^\perp P^\perp,
    \quad\text{and}\quad
    \Gamma_i = \gamma P + \gamma_i^\perp P^\perp.
\end{equation}
With this,  $\tau_i, \tau_i^\perp > 0$ guarantee $T_i \in \mathcal{Q}$.
Moreover, $\Gamma_i \in \lambda([0, \Gamma] + \UTildeSet)$, exactly as required in both the gap and the non-gap cases of Theorem \ref{thm:convergence-result-gap}.

Since
\[
    K\Tau_{i}K^*
    =  \tau_{i} KPK^*
    +\tau_{i}^\perp KP^\perp K^*
    = (\tau_{i} - \tau_{i}^\perp) KPK^*
    + \tau_{i}^\perp KK^*,
\]
we are encouraged to take
\begin{equation}
    \label{eq:both-penalties-ff}
    \mathcal{F}(K \Tau_{i} K^*)
    \defeq
     \max\{0, \tau_{i} - \tau_{i}^\perp\} \norm{KP}^2 I
    + \tau_{i}^\perp \norm{K}^2 I.
\end{equation}

\begin{remark}
If we required $\tau_i^\perp \ge \tau_i$, a simpler choice would be $\mathcal{F}(K \Tau_{i} K^*) = \tau_{i}^\perp \norm{K}^2 I$. Numerical experiments however suggest $\tau_0^\perp \ll \tau_0$ being beneficial. Nevertheless, for large enough $i$, the condition $\tau_i^\perp \ge \tau_i$ will hold in our algorithms.
\end{remark}

Observe that \eqref{eq:both-penalties-ff} satisfies \eqref{eq:ff-lower-bound}. 
Inserting \eqref{eq:both-penalties-ff} into \eqref{eq:tauhatsigma-bound-special-def0}, we obtain
\begin{equation}
    \label{eq:projection-sigma}
    \Sigma_{i+1}=\sigma_{i+1} I
    \quad\text{with}\quad
    \inv\sigma_{i+1} =  \frac{\tilde\omega_i}{1-\delta}\left(\max\{0, \tau_{i} - \tau_{i}^\perp\} \norm{KP}^2 + \tau_i^\perp \norm{K}^2\right).
\end{equation}
Since $\Sigma_{i+1}$ is a scalar,  \eqref{eq:simigamma-simiplus-2-special-simple-2}, we also take $R_{i+1}=\rho_{i+1} I$, assuming for some $\bar\rho>0$ that
\[
    [0, \bar\rho I] \subset \UHatSet.
\]
Setting
\[
    \eta_i 
    \defeq
    \inv{\tilde\tau_i}\max\{0, \tau_{i} - \tau_{i}^\perp\}
    -
    \inv{\tilde\tau_{i+1}}\max\{0, \tau_{i+1} - \tau_{i+1}^\perp\}
\]
we thus expand \eqref{eq:simigamma-simiplus-2-special-simple} as
\begin{subequations}
\label{eq:simigamma-simiplus-2-special-simple-projection}
\begin{align}
    \label{eq:simigamma-simiplus-2-special-simple-projection-1}
    \inv{\tilde\tau_i} (1 + 2 \gamma\tau_i) \inv\tau_i - \tilde\tau_{i+1} \inv\tau_{i+1}
    & \ge 0,
    \\
    \label{eq:simigamma-simiplus-2-special-simple-projection-2}
    \inv{\tilde\tau_i}\tau_i^{\perp,-1} - \inv{\tilde\tau_{i+1}} \tau_{i+1}^{\perp,-1}
    & \ge -2 \inv{\tilde\tau_i} \gamma_i^\perp,
    \\
    \label{eq:simigamma-simiplus-2-special-simple-projection-3}
    \frac{1}{1-\delta} \left(\eta_i \norm{KP}^2+(\inv{\tilde\tau_i} \tau_i^\perp-\inv{\tilde\tau_{i+1}} \tau_{i+1}^\perp)\norm{K}^2\right) 
    & \ge -2 \inv{\tilde\tau_{i+1}} \rho_{i+1}.
\end{align}
\end{subequations}

We are almost ready to state a general convergence result for projective $\Gamma$.
However, we want to make one more thing more explicit.
Since $\Gamma_i-\Gamma=\gamma_i^\perp P^\perp$ and $R_{i+1}=\rho_{i+1} I$, we suppose for simplicity that
\begin{equation}
    \label{eq:phipsi-restrict}
    \phi_{R_{i+1}}(y)=\rho_{i+1} \phi(y)
    \quad\text{and}\quad
    \psi_{\Gamma_i-\Gamma}(x)
    =\gamma_i^\perp\psi^\perp(P^\perp x)
\end{equation}
for some $\phi: Y \to \R$ and $\psi^\perp: P^\perp X \to \R$.
The conditions \eqref{eq:g-strong-convexity-scalar} and \eqref{eq:f-strong-convexity-scalar} reduce in this case to the satisfaction for some $\bar\gamma, \bar\gamma^\perp, \bar\rho>0$ of
\begin{multline}
    \label{eq:g-strong-convexity-scalar-restrict}
    \tag{G-pcr}
    G(x') - G(x) \ge \iprod{\subdiff G(x)}{x'-x} +
        \frac{\bar\gamma}{2}\norm{P(x'-x)}^2+\frac{\gamma^\perp}{2}\left(\norm{P^\perp(x'-x)}^2-\psi(P^\perp({x'-x}))\right),
    \\
    \quad (x, x' \in X;\, 0 \le \gamma^\perp \le \bar\gamma^\perp),
\end{multline}
and
\begin{multline}
    \label{eq:f-convexity-scalar-restrict}
    \tag{F$^*$-pcr}
    F^*(y') - F^*(y) \ge \iprod{\subdiff F^*(y)}{y'-y} +
    \frac{\rho}{2}\left(\norm{y'-y}^2-\phi({y'-y})\right),
    \\
    \quad (y, y' \in Y;\, 0 \le \rho \le \bar\rho).
\end{multline}
Analogous variants of \eqref{eq:g-strong-monotone-scalar} and \eqref{eq:f-monotone-scalar} can be formed.

To summarise the findings of this section, we state the following proposition.

\begin{proposition}
    \label{prop:projective}
    Suppose \eqref{eq:g-strong-convexity-scalar-restrict} and \eqref{eq:f-convexity-scalar-restrict} hold for some projection operator $P \in \L(X; X)$ and scalars $\bar\gamma, \bar\gamma^\perp, \bar\rho > 0$.
    With $\lambda=1/2$, pick $\gamma \in [0, \lambda\bar\gamma]$.
    For each $i \in \N$, suppose \eqref{eq:simigamma-simiplus-2-special-simple-projection} is satisfied with
    \begin{equation}
        \label{eq:prop:projective-step-conds}
        0 \le \gamma_i^\perp \le \lambda\bar\gamma^\perp,
        \quad
        0 \le \rho_i \le \lambda\bar\rho,
        \quad\text{and}\quad
        \tilde\tau_0 \ge \tilde\tau_i >0.
    \end{equation}
    If we solve \eqref{eq:simigamma-simiplus-2-special-simple-projection-1} exactly, define $\Tau_i$, $\Gamma_i$, and $\Sigma_{i+1}$ through \eqref{eq:projection-tau} and \eqref{eq:projection-sigma}, and set $R_{i+1}=\rho_{i+1}I$, then the iterates of Algorithm \ref{alg:alg-scalar} satisfy with $C_0$ and $D_{i+1}$ as in \eqref{eq:diplus1} the estimate
    \begin{equation}
        \label{eq:convergence-result-gap-scalar-projective}
        \frac{\delta}{2}\norm{P(x^N-\realoptx)}^2
        + \frac{1}{\inv\tau_0 + 2\gamma}\ergGap{N}
        \le
        \tilde\tau_N\tau_N\left(
        C_0
        +
        \sum_{i=0}^{N-1} D_{i+1}
        \right).
    \end{equation}

    If we take $\lambda=1$, then \eqref{eq:convergence-result-gap-scalar-projective} holds with $\ergGap{N} = 0$.
\end{proposition}

Observe that presently
\begin{equation}
    \label{eq:diplus1-proj}
    D_{i+1} =
        \inv{\tilde\tau_i}\gamma_i^\perp\psi({x^{i+1}-\realoptx})
        +
        \inv{\tilde\tau_{i+1}}\rho_{i+1}\phi({y^{i+1}-\realopty}).
\end{equation}

\begin{proof}
    As we have assumed through \eqref{eq:prop:projective-step-conds}, or otherwise already verified its conditions, we may apply Theorem \ref{thm:convergence-result-gap}. Multiplying \eqref{eq:convergence-result-gap-scalar} by $\tilde\tau_N\tau_N$, we obtain
    \begin{equation}
        \label{eq:convergence-result-gap-scalar-projective0}
        \frac{\delta}{2}\norm{x^N-\realoptx}_P^2
        + \tilde q_N\tilde\tau_N\tau_N \ergGap{N}
        \le
        \tilde\tau_N\tau_N\biggl(
        C_0
        +
        \sum_{i=0}^{N-1} D_{i+1}
        \biggr).
    \end{equation}
    Now, observe that solving \eqref{eq:simigamma-simiplus-2-special-simple-projection-1} exactly gives
    \begin{equation}
        \label{eq:taun-squared-sum}
        \inv{\tilde\tau_{N}}\inv{\tau_N}
        =
        \inv{\tilde\tau_{N-1}}\inv{\tau_{N-1}}
        + 2\gamma\inv{\tilde\tau_{N-1}}
        =
        \inv{\tilde\tau_{0}}\inv{\tau_{0}}
        +
        \sum_{j=0}^{N-1} 2\gamma\inv{\tilde\tau_{j}}
        =
        \inv{\tilde\tau_{0}}\inv{\tau_{0}}
        + 2\gamma \tilde q_N.
    \end{equation}
    Therefore, we have the estimate
    \begin{equation}
        \label{eq:gap-factor-estimate}
        \tilde q_N\tilde\tau_{N}\tau_N
        =
        \frac{\tilde q_N}{\inv{\tilde\tau_{0}}\inv{\tau_{0}} + 2\gamma \tilde q_N}
        =
        \frac{1}{\inv{\tilde\tau_{0}}\inv{\tau_{0}} \inv{\tilde q_N}+ 2\gamma}
        \ge
        \frac{1}{\inv{\tau_{0}} + 2\gamma}.
    \end{equation}
    With this, \eqref{eq:convergence-result-gap-scalar-projective0} yields \eqref{eq:convergence-result-gap-scalar-projective}.
\end{proof}

\subsection{Primal and dual penalties with projective $\Gamma$}

We now study conditions that guarantee the convergence of the sum $\tilde\tau_N\tau_N \sum_{i=0}^{N-1} D_{i+1}$ in \eqref{eq:convergence-result-gap-scalar-projective}.
Indeed, the right-hand-sides of \eqref{eq:simigamma-simiplus-2-special-simple-projection-2} and \eqref{eq:simigamma-simiplus-2-special-simple-projection-3} relate to $D_{i+1}$. 
In most practical cases, which we study below, $\phi$ and $\psi$ transfer these right-hand-side penalties into simple linear factors within $D_{i+1}$. Optimal rates are therefore obtained by solving \eqref{eq:simigamma-simiplus-2-special-simple-projection-2} and \eqref {eq:simigamma-simiplus-2-special-simple-projection-3} as equalities, with the right-hand-sides proportional to each other.
Since $\eta_i \ge 0$, and it will be the case that $\eta_i=0$ for large $i$, we however replace \eqref{eq:simigamma-simiplus-2-special-simple-projection-3} by the simpler condition
\begin{equation}
    \label{eq:simigamma-simiplus-2-special-simple-projection-4}
    \frac{1}{1-\delta} (\inv{\tilde\tau_i} \tau_i^\perp-\inv{\tilde\tau_{i+1}} \tau_{i+1}^\perp)\norm{K}^2 
    \ge -2 \inv{\tilde\tau_{i+1}} \rho_{i+1}.
\end{equation}
Then we try to make the left hand sides of \eqref{eq:simigamma-simiplus-2-special-simple-projection-2} and  \eqref{eq:simigamma-simiplus-2-special-simple-projection-4} proportional with only $\tau_{i+1}^\perp$ as a free variable. That is, for some proportionality constant $\zeta > 0$, we solve
\begin{equation}
    \label{eq:prop-balancing}
    \inv{\tilde\tau_i}\tau_i^{\perp,-1} - \inv{\tilde\tau_{i+1}} \tau_{i+1}^{\perp,-1}
    = \zeta (\inv{\tilde\tau_i} \tau_i^\perp-\inv{\tilde\tau_{i+1}} \tau_{i+1}^\perp).
\end{equation}
Multiplying both sides of \eqref{eq:prop-balancing} by  $\inv\zeta \tilde\tau_{i+1}\tau_{i+1}^\perp$, gives on $\tau_{i+1}^\perp$ the quadratic condition
\[
    \tau_{i+1}^{\perp,2}
    +\tilde\omega_i(\inv\zeta\tau_i^{\perp,-1}- \tau_i^\perp) \tau_{i+1}^\perp
    - \inv\zeta
    = 0.
\]
Thus
\begin{equation}
    \label{eq:both-penalties-tauperp-update}
    \tau_{i+1}^\perp
    =
    \frac{1}{2}
    \left(
        \tilde\omega_i(\tau_i^\perp-\inv{\zeta}\tau_i^{\perp,-1})
        + \sqrt{\tilde\omega_i^2(\tau_i^\perp-\inv{\zeta}\tau_i^{\perp,-1})^2+4\inv{\zeta}}
    \right).
\end{equation}

Solving \eqref{eq:simigamma-simiplus-2-special-simple-projection-2} and  \eqref{eq:simigamma-simiplus-2-special-simple-projection-4} as equalities, \eqref{eq:prop-balancing} and \eqref{eq:both-penalties-tauperp-update} give
\begin{equation}
    \label{eq:both-penalties-nonnegativity}
    2\inv{\tilde\tau_i} \gamma_i^\perp
    =
    \frac{2\zeta(1-\delta)}{\norm{K}^2} \inv{\tilde\tau_{i+1}} \rho_{i+1}
    =
    \zeta (\inv{\tilde\tau_{i+1}} \tau_{i+1}^\perp-\inv{\tilde\tau_i} \tau_i^\perp).
\end{equation}
Note that this quantity is non-negative exactly when $\omega_i^\perp \ge \tilde\omega_i$.
We have
\[
    \frac{\omega_i^\perp}{\tilde\omega_i}
    =
    \frac{\tau_{i+1}^\perp}{\tau_i^\perp\tilde\omega_i}
    =
    \frac{1}{2}
    \left(
        1-\inv{\zeta}\tau_i^{\perp,-2}
        + \sqrt{(1-\inv{\zeta}\tau_i^{\perp,-2})^2+4\inv{\zeta}\tilde\omega_i^{-2}\tau_i^{\perp,-2}}
    \right).
\]
Thus $\omega_i^\perp \ge \tilde\omega_i$ if
\[
    (1-\inv{\zeta}\tau_i^{\perp,-2})^2+4\inv{\zeta}\tilde\omega_i^{-2}\tau_i^{\perp,-2} \ge (1+\inv{\zeta}\tau_i^{\perp,-2})^2.
\]
This gives the condition $\inv{\zeta}\tilde\omega_i^{-2}\tau_i^{\perp,-2} \ge \inv{\zeta}\tau_i^{\perp,-2}$, which says that \eqref{eq:both-penalties-nonnegativity} is non-negative when $\tilde\omega_i\le 1$.

The next lemma summarises these results for the standard choice of $\tilde\omega_i$.

\begin{lemma}
    \label{lemma:tauperp-rule}
    Let $\tau_{i+1}^\perp$ by given by \eqref{eq:both-penalties-tauperp-update}, and set
    \begin{equation}
        \label{eq:tildeomega=omega}
        \tilde\omega_i=\omega_i=1/\sqrt{1+2\gamma\tau_i}.
    \end{equation}
    Then $\omega_i^\perp \ge \tilde\omega_i$, 
    $\tilde\tau_i \le \tilde\tau_0$,
    and \eqref{eq:simigamma-simiplus-2-special-simple-projection} is satisfied with the right-hand-sides given by the non-negative quantity in \eqref{eq:both-penalties-nonnegativity}. Moreover,
    \begin{equation}
        \label{eq:taui-tauip1-perp-bound}
        \tau_i^\perp \le \zeta^{-1/2}
        \implies
        \tau_{i+1}^\perp \le \zeta^{-1/2}.
    \end{equation}
\end{lemma}

\begin{proof}
    The choice \eqref{eq:tildeomega=omega} satisfies \eqref{eq:simigamma-simiplus-2-special-simple-projection-1}, so that \eqref{eq:simigamma-simiplus-2-special-simple-projection} in its entirety will be satisfied with the right-hand sides of \eqref{eq:simigamma-simiplus-2-special-simple-projection-2}--\eqref{eq:simigamma-simiplus-2-special-simple-projection-3} given by \eqref{eq:both-penalties-nonnegativity}. 
    The bound 
    $\tilde\tau_i \le \tilde\tau_0$ follows from $\tilde\omega_i \le 1$.
    Finally, the implication \eqref{eq:taui-tauip1-perp-bound} is a simple estimation of \eqref{eq:both-penalties-tauperp-update}.
\end{proof}

Specialisation of Algorithm \ref{alg:alg-scalar} to the choices in Lemma \ref{lemma:tauperp-rule} yields the steps of Algorithm \ref{alg:alg-projective-both}. Observe that $\tilde\tau_i$ entirely disappears from the algorithm. To obtain convergence rates, and to justify the initial conditions, we will shortly seek to exploit with specific $\phi$ and $\psi$ the telescoping property stemming from the non-negativity of the last term of \eqref{eq:both-penalties-nonnegativity}.

\begin{Algorithm}
    \caption{Partial acceleration for projective $\Gamma$---primal and dual penalties}
    \label{alg:alg-projective-both}
    \begin{subequations}
    \begin{algorithmic}[1]
    \REQUIRE 
    $F^*$ and $G$ satisfying \eqref{eq:g-strong-convexity-scalar-restrict} and \eqref{eq:f-convexity-scalar-restrict} for some $\bar\gamma, \bar\gamma^\perp, \bar\rho \ge 0$, and a projection operator $P \in \L(X; X)$.
    A choice of $\gamma \in [0, \bar\gamma]$.
    Initial step length parameters $\tau_0, \tau^\perp_0 > 0$, a choice of $\delta \in (0, 1)$, and $\zeta \le \tau_0^{\perp, -2}$, all satisfying \eqref{eq:tildec-condition}.
    \STATE Choose initial iterates $x^0 \in X$ and $y^0 \in Y$.
    \REPEAT
    \STATE Set
    \begin{align}
        \notag
        \omega_i & =1/\sqrt{1+2\gamma\tau_i},
        \quad\text{and}
         \\
        \notag
        \omega_i^\perp
        & =
        \frac{1}{2}
        \left(
            (1-\inv{\zeta}\tau_i^{\perp,-2})\omega_i
            + \sqrt{(1-\inv{\zeta}\tau_i^{\perp,-2})^2\omega_i^2+4\inv{\zeta}\tau_i^{\perp,-2}}
        \right).
    \end{align}
    \STATE Update
    \begin{align}
        \notag%
        \tau_{i+1}& =\tau_i\omega_i, \quad
        \tau_{i+1}^\perp =\tau_i^\perp\omega_i^\perp,
        \quad\text{and}\\
        \notag
        \sigma_{i+1} & =  \inv\omega_i(1-\delta)/\left(\max\{0, \tau_{i} - \tau_{i}^\perp\} \norm{KP}^2 + \tau_{i}^\perp \norm{K}^2\right),
    \end{align}
    \STATE With $T_i=\tau_i P + \tau_i^\perp P^\perp$, 
    perform the updates
    \label{eq:alg-projective-both}
    \begin{align}
        \label{eq:alg-projective-both-nextx}
        \nextx & \defeq (I+\Tau_i \subdiff G)^{-1}(\thisx - \Tau_i K^* y^{i}),
        \\
        \label{eq:alg-projective-both-overnextx}
        \overnextx & \defeq \omega_i (\nextx-\thisx)+\nextx,
        \\
        \label{eq:alg-projective-both-nexty}
        \nexty & \defeq (I+\sigma_{i+1} \subdiff F^*)^{-1}(\thisy + \sigma_{i+1} K \overnextx).
    \end{align}   
    \UNTIL a stopping criterion is fulfilled.
    \end{algorithmic}
    \end{subequations}
\end{Algorithm}

There is still, however, one matter to take care of. We need $\rho_i \le \lambda\bar\rho$ and $\gamma_i^\perp \le \lambda\bar\gamma^\perp$, although in many cases of practical interest, the upper bounds are infinite and hence inconsequential.
We calculate from \eqref{eq:both-penalties-tauperp-update} and \eqref{eq:tildeomega=omega} that
\begin{equation}
    \label{eq:both-penalties-gamma-bound-verification}
    \begin{split}
    \gamma_i^\perp
    =
    \frac{\zeta}{2}
    (\inv{\tilde\omega_i} \tau_{i+1}^\perp- \tau_i^\perp)
    &
    =
    \frac{1}{2}
    \left(
        -\zeta\tau_i^\perp-\tau_i^{\perp,-1}
        + \sqrt{(\zeta\tau_i^\perp-\tau_i^{\perp,-1})^2+4\zeta\tilde\omega_i^{-2}}
    \right)
    \\
    &
    \le
    \frac{1}{2}
        \sqrt{(\zeta\tau_i^\perp-\tau_i^{\perp,-1})^2 -(\zeta\tau_i^\perp+\tau_i^{\perp,-1})^2 +4\zeta\tilde\omega_i^{-2}}
    \\
    &
    =\sqrt{\zeta(\tilde\omega_i^{-2}-1)}
    =\sqrt{2\zeta \gamma \tau_i}
    \le\sqrt{2\zeta\gamma\tau_0}.
    \end{split}
\end{equation}
Therefore, we need to choose $\zeta$ and $\tau_0$ to satisfy $2\zeta\gamma\tau_0 \le (\lambda\bar\gamma^\perp)^2$. Likewise, we calculate from \eqref{eq:both-penalties-nonnegativity}, \eqref{eq:tildeomega=omega}, and \eqref{eq:both-penalties-gamma-bound-verification} that
\[
    \rho_{i+1}
    =
    \frac{\tilde \omega_i}{c} \gamma_i^\perp
    =
    \frac{\norm{K}^2 \tilde \omega_i }{(1-\delta)\zeta} \gamma_i^\perp
    \le 
    \frac{\norm{K}^2 \tilde \omega_i }{(1-\delta)\zeta} \sqrt{2 \zeta \gamma \tau_i}
    =
    \frac{\norm{K}^2}{(1-\delta)\zeta} \sqrt{2 \zeta \gamma \tau_0}.
\]
This tells us to choose $\tau_0$ and $\zeta$ to satisfy $2 \norm{K}^4/(1-\delta)^2 \zeta^{-1} \gamma \tau_0 \le (\lambda\bar\rho)^2$.
Overall, we get on $\tau_0$ and $\zeta$ the always satisfiable condition
\begin{equation}
    \label{eq:tildec-condition}
    0 < \tau_0 
    \le
    \frac{\lambda^2}{2\gamma}
    \min\left\{
        \frac{\bar\gamma^{\perp,2}}{\zeta},
        \frac{\bar\rho^2 \zeta (1-\delta)^2}{\norm{K}^4}
    \right\}.
\end{equation}

If now $\phi \equiv C_\phi$ and $\psi \equiv C_\psi^\perp$, using the non-negativity of \eqref{eq:both-penalties-nonnegativity} when $0 < \tilde\omega_i \le 1$, we may calculate
\begin{equation}
    \label{eq:projective-both-phisum}
    \sum_{i=0}^{N-1} \inv{\tilde\tau}_{i+1}\rho_{i+1} \phi({y^{i+1}-\realopty})
    =
    \frac{\norm{K}^2C_\phi}{2(1-\delta)}\left(
    \sum_{i=0}^{N-1}
        \frac{\inv{\tilde\tau_{i+1}} \tau_{i+1}^\perp}{2}
    -
    \sum_{i=0}^{N-1}
        \frac{\inv{\tilde\tau_i} \tau_i^\perp}{2}
    \right)
    \le
        \frac{\norm{K}^2 C_\phi}{2(1-\delta)}
        \inv{\tilde\tau_{N}} \tau_{N}^\perp.
\end{equation}
Similarly
\begin{equation}
    \label{eq:projective-both-psisum}
    \sum_{i=0}^{N-1} \inv{\tilde\tau}_i \gamma_i^\perp \psi({x^{i+1}-\realoptx})
    \le
        \frac{\zeta C_\psi^\perp}{2} \inv{\tilde\tau_{N}} \tau_{N}^\perp.
\end{equation}
Using these expression to expand \eqref{eq:diplus1-proj}, we obtain the following convergence result.

\begin{theorem}
    \label{thm:projective-both}
    Suppose \eqref{eq:g-strong-convexity-scalar-restrict} and \eqref{eq:f-convexity-scalar-restrict} hold for some projection operator $P \in \L(X; X)$, scalars $\bar\gamma, \bar\gamma^\perp, \bar\rho > 0$,
    and
    \[
        \phi \equiv  C_\phi,
        \quad\text{and}\quad
        \psi \equiv C_\psi^\perp
    \]
    for some constants $C_\phi, C_\psi^\perp >0$.
    With $\lambda=1/2$, fix $\gamma \in (0, \lambda\gamma]$.
    Select initial $\tau_0, \tau^\perp_0 > 0$, as well as $\delta \in (0, 1)$ and $\zeta \le (\tau_0^\perp)^{-2}$ satisfying \eqref{eq:tildec-condition}.
    Then Algorithm \ref{alg:alg-projective-both} satisfies for some $C_0,C_\tau>0$ the estimate
    \begin{equation}
        \label{eq:eq:convergence-result-gap-scalar-telescope-const2}
        \frac{\delta}{2}\norm{P(x^N-\realoptx)}^2
        + \frac{1}{\inv\tau_0+2\gamma}\ergGap{N}
        \le
        \frac{C_0 C_\tau^2}{N^2}
        +
        \frac{C_\tau}{2 N}\left(\zeta^{1/2} C_\psi^\perp+\frac{\zeta^{-1/2}\norm{K}^2}{1-\delta}C_\phi\right),
        \quad (N \ge 0).
    \end{equation}

    If we take $\lambda=1$, then \eqref{eq:convergence-result-gap-scalar-projective} holds with $\ergGap{N} = 0$.
\end{theorem}

\begin{proof}
    During the course of the derivation of Algorithm \ref{alg:alg-projective-both}, we have verified \eqref{eq:simigamma-simiplus-2-special-simple-projection}, solving \eqref{eq:simigamma-simiplus-2-special-simple-projection-1} as an equality.
    Moreover, Lemma \ref{lemma:tauperp-rule} and \eqref{eq:tildec-condition} guarantee \eqref{eq:prop:projective-step-conds}. We may therefore apply Proposition \ref{prop:projective}.
    Inserting \eqref{eq:projective-both-phisum} and \eqref{eq:projective-both-psisum} into \eqref{eq:convergence-result-gap-scalar-projective} and \eqref{eq:diplus1-proj} gives
    \begin{equation}
        \label{eq:convergence-result-gap-scalar-telescope-pre}
        \frac{\delta}{2}\norm{P(x^N-\realoptx)}^2
        + \frac{1}{\inv\tau_0+2\gamma}\ergGap{N}
        \le
        \tau_N\tilde\tau_N
        \left(
            C_0
            +
            \frac{\zeta C_\psi^\perp}{2} \inv{\tilde\tau_{N}} \tau_{N}^\perp
            +
            \frac{\norm{K}^2 C_\phi}{2(1-\delta)}
            \inv{\tilde\tau_{N}} \tau_{N}^\perp
        \right).
    \end{equation}
    The condition $\zeta \le (\tau_0^\perp)^{-2}$ now guarantees $\tau_N^\perp \le \zeta^{-1/2}$ through \eqref{eq:taui-tauip1-perp-bound}.
    Now we note that $\tilde\tau_i$ is not used in Algorithm \ref{alg:alg-projective-both}, so it only affects the convergence rate estimates. We therefore simply take $\tilde\tau_0=\tau_0$, so that $\tilde\tau_N=\tau_N$ for all $N \in \N$. With this and the bound $\tau_N \le C_\tau/N$ from Remark \ref{rem:improved-tau-estimate}, \eqref{eq:eq:convergence-result-gap-scalar-telescope-const2} follows by simple estimation of \eqref{eq:convergence-result-gap-scalar-telescope-pre}.
\end{proof}

\begin{remark}
    As a special case of Algorithm \ref{alg:alg-projective-both}, if we choose $\zeta = \tau_0^{\perp, -2}$, then we can show from \eqref{eq:both-penalties-tauperp-update} that $\tau_i^\perp=\tau_0^\perp=\zeta^{-1/2}$ for all $i \in \N$.
\end{remark}

\subsection{Dual penalty only with projective $\Gamma$}

Continuing with the projective $\Gamma$ setup of Section \ref{sec:projective}, we now study the case $\UTildeSet=\{0\}$, that is, when only the dual penalty $\phi$ is available with $\psi \equiv 0$.
To use Proposition \ref{prop:projective}, we need to satisfy \eqref{eq:prop:projective-step-conds} and \eqref{eq:simigamma-simiplus-2-special-simple-projection}, with \eqref{eq:simigamma-simiplus-2-special-simple-projection-1} exactly. Since $\gamma_i^\perp =0$, \eqref{eq:simigamma-simiplus-2-special-simple-projection-2} becomes 
\begin{equation}
    \label{eq:simigamma-simiplus-2-special-simple-projection-2-dualonly} 
    \inv{\tilde\tau_i}\tau_i^{\perp,-1} - \inv{\tilde\tau_{i+1}} \tau_{i+1}^{\perp,-1} \ge 0.
\end{equation}
With respect to $\tau_{i+1}^{\perp}$, the left hand side of \eqref{eq:simigamma-simiplus-2-special-simple-projection-3}  is maximised (and the penalty on the right hand side minimised) when \eqref{eq:simigamma-simiplus-2-special-simple-projection-2-dualonly} is minimised. Thus we solve \eqref{eq:simigamma-simiplus-2-special-simple-projection-2-dualonly} exactly, which gives
\[
    \tau_{i+1}^{\perp}= \tau_i^{\perp,}\inv{\tilde\omega_i}.
\]
In consequence $\omega_i^\perp=\inv{\tilde\omega_i}$, and \eqref{eq:simigamma-simiplus-2-special-simple-projection-3} becomes
\begin{equation}
    \label{eq:dual-penalty-rho}
    \frac{1}{1-\delta}\eta_i \norm{KP}^2
    +
    \frac{\tilde\tau_i^{-2}}{1-\delta} (1-\tilde\omega_{i}^{-2})\norm{K}^2 
    \ge -2 \inv{\tilde\tau_{i+1}} \rho_{i+1}.
\end{equation}
Here $\eta_i \ge 0$, so we estimate this as $\eta_i=0$ as in \eqref{eq:simigamma-simiplus-2-special-simple-projection-4}.
This suggests to choose
\begin{equation}
    \label{eq:gamma-projection-omega-choice}
    \tilde\omega_{i} \defeq \frac{1}{1+a_i\tilde\tau_i^2}
    \quad\text{and}\quad
    \omega_i \defeq \frac{1}{\tilde\omega_i(1+2\gamma\tau_i)},
\end{equation}
for some, yet undetermined, $a_i>0$. Solving \eqref{eq:dual-penalty-rho} as an equality for $\rho_{i+1}$, then
\[
    2 \inv{\tilde\tau_{i+1}} \rho_{i+1}
    =a_i\frac{\norm{K}^2}{1-\delta}.
\]
This needs $\rho_{i+1} \le \lambda\bar\rho$.
Since $\inv{\tilde\tau_i} \ge \inv{\tilde\tau_0}$, we can satisfy this for large enough $i$ if $a_i \downto 0$, or generally if $\tilde\tau_0$ is small enough and $\{a_i\}$ non-decreasing. In particular, if $\{a_i\}$ is descending, it suffices
\begin{equation}
    \label{eq:dual-penalty-cond}
    a_0 \tau_0^\perp \tilde\tau_0^2 \frac{\norm{K}^2}{2(1-\delta)} \le \bar\rho
\end{equation}

\def\residual{\mu}

Noting that \eqref{eq:gamma-projection-omega-choice} ensures $\tilde\tau_{i+1}^{-2} = \tilde\tau_{i}^{-2} + a_i$, we see that
\[
    \inv{\tilde\tau_N}\inv{\tau_N}
    =\inv{\tilde\tau_0}\inv{\tau_0} + 2\gamma\sum_{i=0}^{N-1} \sqrt{\tilde\tau_0^{-2} + \sum_{j=0}^{i-1} a_j}
    \ge
    2\gamma\sum_{i=0}^{N-1} \sqrt{\tilde\tau_0^{-2} + \sum_{j=0}^{i-1} a_j}
    =: 1/\residual_0^N.
\]
Assuming $\phi$ to have the structure \eqref{eq:phipsi-restrict}, moreover
\[
        \sum_{i=0}^{N-1} D_{i+1} = \sum_{i=0}^{N-1}
            \phi_{\inv{\tilde\tau_{i+1}} R_{i+1}}({y^{i+1}-\realopty})
        =
        \frac{\norm{K}^2}{2(1-\delta)}
        \sum_{i=0}^{N-1} a_i \phi({y^{i+1}-\realopty}).
\]
Thus the rate \eqref{eq:convergence-result-gap-scalar-projective} in Proposition \ref{prop:projective} states
\begin{equation}
    \label{eq:convergence-result-gap-scalar-projective-dual}
    \frac{\delta}{2}\norm{P(x^N-\realoptx)}^2
    + \frac{1}{\inv\tau_0 + 2\gamma}\ergGap{N}
    \le
    \residual_0^N C_0
    +
    \frac{\norm{K}^2}{2(1-\delta)}
    \residual_1^N
\end{equation}
for
\[
    \residual_1^N \defeq \residual_0^N \sum_{i=0}^{N-1} a_i \phi({y^{i+1}-\realopty}).
\]
The convergence rate is thus completely determined by $\residual_0^N$ and $\residual_1^N$.


\begin{remark}
If $\phi \equiv 0$, that is, if $F^*$ is strongly convex, we may simply pick $\tilde\omega_i=\omega_i=1/\sqrt{1+2\gamma\tau_i}$, that is $a_i=2\gamma$, and obtain from \eqref{eq:convergence-result-gap-scalar-projective-dual} a $O(1/N^2)$ convergence rate.
\end{remark}


For a more generally applicable algorithm,
suppose $\phi({\nexty-\realopty}) \equiv C_\phi$ as in Theorem \ref{thm:projective-both}. 
We need to choose $a_i$.
One possibility is to pick some $q>0$ and
\begin{equation}
    \label{eq:ai-choice}
    a_i \defeq \tilde\tau_0^{-2}\bigl((i+1)^{q}-i^{q}\bigr).
\end{equation}
This gives
\[
    \sum_{i=0}^{N-1} \sqrt{\tilde\tau_0^{-2} + \sum_{j=0}^{i-1} a_j}
    =\tilde\tau_0^{-1}\sum_{i=0}^{N-1} i^{q/2}
    \ge
    \tilde\tau_0^{-1}
    \int_0^{N-1} x^{q/2} \d x
    = \frac{\tilde\tau_0^{-1}}{1+q/2}(N-1)^{1+q/2},
\]
and
\[
    \sum_{i=0}^{N-1} a_i \le \tilde\tau_0^{-2} N^q.
\]
If $N \ge 2$, we find with $C_a=(1+q/2)/(2^{1+q/2}\lambda\gamma)$ that
\begin{equation}
    \label{eq:dual-only-rates}
    \residual_0^N \le \frac{\tilde\tau_0 C_a}{N^{1+q/2}},
    \quad\text{and}\quad
    \residual_1^N \le \frac{C_a C_\phi}{\tilde\tau_0N^{1-q/2}}.
\end{equation}
The choice $q=0$ gives uniform $O(1/N)$ over both the initialisation and the dual sequence. By choosing $q<2$ large, we can get arbitrarily close to $O(1/N^2)$ rate with respect to the initialisation, at the cost of the rate $\residual_1^N$ with respect to the dual sequence becoming closer and closer to zero.

With these choices, Algorithm \ref{alg:alg-scalar} yields Algorithm \ref{alg:alg-proj}, whose convergence properties are stated in the next theorem.

\begin{Algorithm}
    \caption{Partial acceleration for projective $\Gamma$---dual penalty only}
    \label{alg:alg-proj}
    \begin{subequations}
    \begin{algorithmic}[1]
    \REQUIRE 
    $G$ satisfying \eqref{eq:g-strong-convexity-scalar-restrict} (with $\psi \equiv 0$) for some $\bar\gamma>0$ and a projection operator $P \in \L(X; X)$.
    $F^*$ satisfying \eqref{eq:f-convexity-scalar-restrict} for some $\bar\rho>0$.
    A choice of $\gamma \in [0, \bar\gamma]$ and a decreasing sequence $\{a_i\}_{i=0}^\infty$, for example as in \eqref{eq:ai-choice}.
    Initial step parameters $\tau_0, \tau_0^\perp, \tilde \tau_0 >0$, as well as $\delta \in (0, 1)$, satisfying \eqref{eq:dual-penalty-cond}.
    \STATE Choose initial iterates $x^0 \in X$ and $y^0 \in Y$.
    \REPEAT
    \STATE Set
        \[
            \begin{aligned}
            \tilde \omega_i & \defeq 1/(1+a_i \tilde\tau_i^2),
            &
            \tilde\tau_{i+1} & \defeq \tilde\tau_i\tilde\omega_i,
            &
            \tau_{i+1}^\perp & \defeq \tau_i^\perp/\tilde\omega_i,
            \\
            \omega_i & \defeq \inv{\tilde\omega_i}/(1+ 2 \gamma \tau_i),
            &
            \tau_{i+1} & \defeq \tau_i\omega_i,
            \end{aligned}
        \]
        as well as
        \[
            \sigma_{i+1} =  \inv\omega_i(1-\delta)/\left(\max\{0, \tau_{i} - \tau_{i}^\perp\} \norm{KP}^2 + \tau_{i}^\perp \norm{K}^2\right).
        \]

    \STATE With $\Tau_i \defeq \tau_i P + \tau_i^\perp P^\perp$, perform the updates
    \begin{align}
        \notag
        \nextx & \defeq (I+\Tau_i \subdiff G)^{-1}(\thisx - \Tau_i K^* y^{i}),\\
        \notag
        \overnextx & \defeq \tilde\omega_i(\nextx-\thisx)+\nextx,
        \\
        \notag
        \nexty & \defeq (I+\sigma_{i+1} \subdiff F^*)^{-1}(\thisy + \sigma_{i+1} K \overnextx).
    \end{align}   
    \UNTIL a stopping criterion is fulfilled.
    \end{algorithmic}
    \end{subequations}
\end{Algorithm}

\begin{theorem}
    \label{thm:projective-dualonly}
    Suppose \eqref{eq:g-strong-convexity-scalar-restrict} and \eqref{eq:f-convexity-scalar-restrict} hold for some projection operator $P \in \L(X; X)$ and  $\bar\gamma, \bar\gamma^\perp, \bar\rho \ge 0$ with $\psi \equiv 0$ and $\phi \equiv C_\phi$ for some constant $C_\phi \ge 0$.
    With $\lambda=1/2$, choose $\gamma \in (0, \lambda\bar\gamma]$, and pick the sequence $\{a_i\}_{i=0}^\infty$ by \eqref{eq:ai-choice} for some $q>0$. Select initial $\tau_0, \tau_0^\perp, \tilde \tau_0 >0$ and $\delta \in (0, 1)$ verifying \eqref{eq:dual-penalty-cond}.
    Then Algorithm \ref{alg:alg-proj} satisfies
    \begin{equation}
        \label{eq:convergence-result-gap-scalar-projective-dual-final}
        \frac{\delta}{2}\norm{P(x^N-\realoptx)}^2
        + \frac{1}{\inv\tau_0 + \gamma}\ergGap{N}
        \le
        \frac{\tilde\tau_0 C_a C_0}{N^{1+q/2}}
        +
        \frac{C_a C_\phi \norm{K}^2}{2(1-\delta)\tilde\tau_0^2 N^{1-q/2}},
        \quad (N \ge 2).
    \end{equation}

    If we take $\lambda=1$, then \eqref{eq:convergence-result-gap-scalar-projective-dual-final} holds with $\ergGap{N} = 0$.
\end{theorem}

\begin{proof}
    We apply Proposition \ref{prop:projective} whose assumptions we have verified during the course of the present section.
    In particular, $\tilde\tau_i \le \tilde\tau_0$ through the choice  \eqref{eq:gamma-projection-omega-choice} that forces $\tilde\omega_i \le 1$.
    Also, have already derived the rate  \eqref{eq:convergence-result-gap-scalar-projective-dual} from \eqref{eq:convergence-result-gap-scalar-projective}. Inserting \eqref{eq:dual-only-rates} into \eqref{eq:convergence-result-gap-scalar-projective-dual}, noting that the former is only valid for $N \ge 2$, immediately gives  \eqref{eq:convergence-result-gap-scalar-projective-dual-final}
\end{proof}

\section{Examples from image processing and the data sciences}
\label{sec:example}

We now consider several applications of our algorithms.
We generally have to consider discretisations, since many interesting infinite-dimensional problems necessitate Banach spaces. Using Bregman distances, it would be possible to generalise our work form Hilbert spaces to Banach spaces, as was done in \cite{hohage2014generalization} for the original method of \cite{chambolle2010first}. This is however outside the scope of the present work.

\begin{figure}
    \centering
    \begin{subfigure}{0.31\textwidth}
        \includegraphics[width=\textwidth]{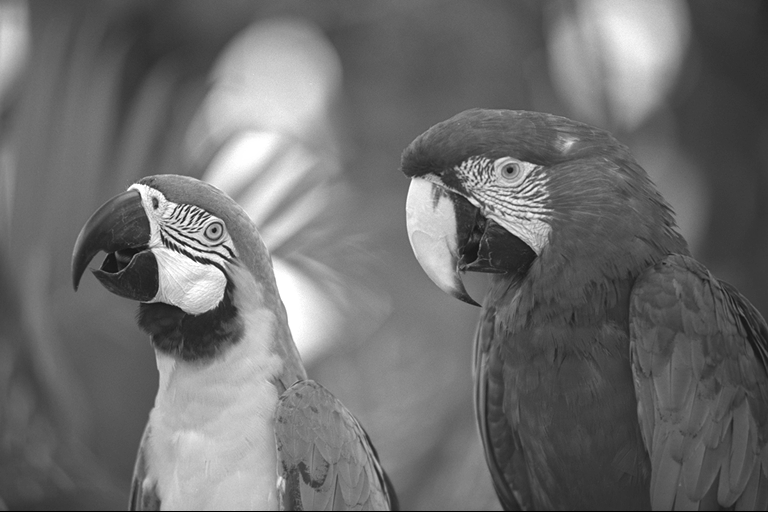}
        \caption{True image}
    \end{subfigure}
    \begin{subfigure}{0.31\textwidth}
        \includegraphics[width=\textwidth]{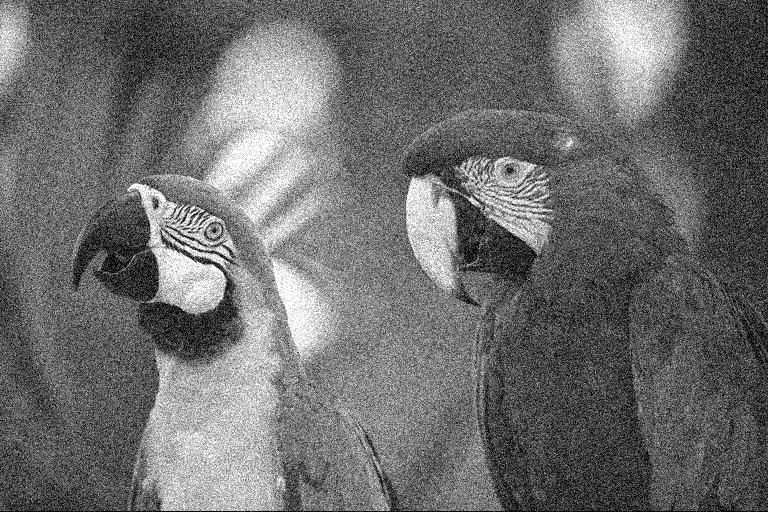}
        \caption{Noisy image}
        \label{fig:noisy}
    \end{subfigure}
    \begin{subfigure}{0.31\textwidth}
        \includegraphics[width=\textwidth]{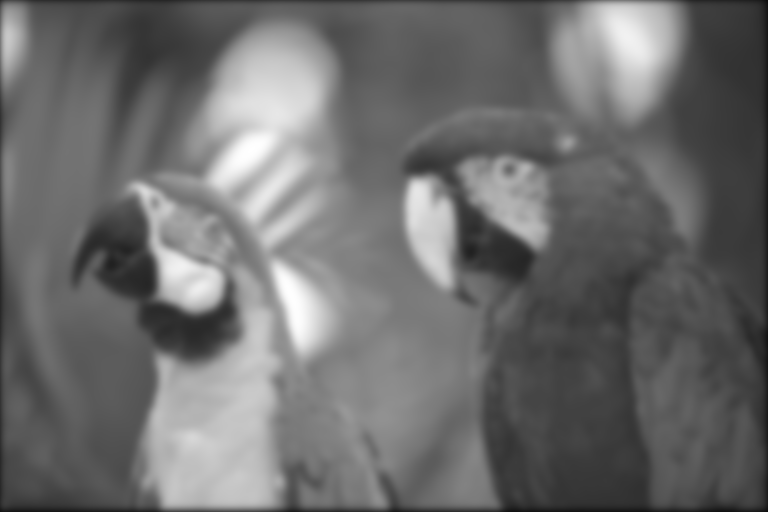}
        \caption{Blurry image}
        \label{fig:blurry}
    \end{subfigure}
    \caption{We use sample image (\subref{fig:noisy}) for denoising, and (\subref{fig:blurry}) for deblurring experiments.
    \emph{\footnotesize
    Free Kodak image suite photo, at the time of writing online at \url{http://r0k.us/graphics/kodak/}.}
    } 
    \label{fig:image}
\end{figure}

\subsection{Regularised least squares}
\label{sec:rlsq}

A large range of interesting application problems can be written in the \emph{Tikhonov regularisation} or \emph{empirical loss minimisation} form
\begin{equation}
    \label{eq:problem-rlsq}
    \min_{x \in X}~ G_0(f-Ax) + \alpha F(Kx).
\end{equation}
Here $\alpha>0$ is a regularisation parameter, $G_0: Z \to \R$ typically convex and smooth fidelity term with data $f \in Z$. The forward operator $A \in \linear(X; Z)$---which can often also be data---maps our unknown to the space of data. The operator $K \in \linear(X; Y)$ and the typically non-smooth and convex $F:  Y \to \extR$ act as a regulariser, although in case of support vector machines, for example, the smooth function is the regulariser.

We are particularly interested in strongly convex $G_0$ and $A$ with a non-trivial null-space. Examples include, for example Lasso---a type of regularised regression---with $G_0=\norm{x}_2^2/2$, $K=I$, and $F(x)=\norm{x}_1$, on finite-dimensional spaces. 
If the data of the Lasso is ``sparse'', in the sense that $A$ has a non-trivial null-space, then our algorithm can provide improved convergence rates.

In image processing examples abound; we refer to \cite{chan2005image} for an overview. In total variation ($\TV$) regularisation we still take $F(x)=\norm{x}_1$, but $K=\grad$.
Strictly speaking, this has to be formulated in the Banach space $\BVspace(\Omega)$, but we will consider the discretised setting to avoid this problem. 
For denoising of Gaussian noise with $\TV$ regularisation, we take $A=I$, and again $G_0=\norm{x}_2^2/2$. This problem is not so interesting to us, as it is fully strongly convex. In a simple form of $\TV$ inpainting---filling in missing regions of an image---we take $A$ as a sub-sampling operator $S$ mapping an image $x \in L^2(\Omega)$ to one in $L^2(\Omega \setminus \Omega_d)$, for $\Omega_d \subset \Omega$ the defect region that we want to recreate. Observe that in this case, $\Gamma=S^*S$ is directly a projection operator. This is therefore a problem for our algorithms! Related problems include reconstruction from subsampled magnetic resonance imaging (MRI) data (see, e.g., \cite{tuomov-phaserec,benning2015preconditioned}), where we take $A=S\mathcal{F}$ for $\mathcal{F}$ the Fourier transform. Still, $A^*A$ is a projection operator, so the problem perfectly suits our algorithms.

Another related problem is total variation deblurring, where $A$ is a convolution kernel. This problem is slightly more complicated to handle, as $A^*A$ is not a projection operator.Assuming periodic boundary conditions on a box $\Omega=\prod_{i=1}^m [c_i, d_i]$, we can write $A=\mathcal{F}^* \hat a \mathcal{F}$, multiplying the Fourier transform by some $\hat a \in L^2(\Omega)$. If $\abs{\hat a} \ge \gamma$ on a sub-domain, we obtain a projection-form $\Gamma$. (It would also be possible to extend our theory to non-constant $\gamma$, but we have decided not to extend the length of the paper by doing so. Dualisation likewise provides a further alternative.)

\paragraph{Satisfaction of convexity conditions}

In all of the above examples, when written in the saddle point form \eqref{eq:problem}, $F^*$ is a simple pointwise ball constraint. Lemma \ref{lemma:bounded-strong-convexity-constant} thus guarantees \eqref{eq:f-convexity-scalar-restrict}.
If $F(x)=\norm{x}_1$ and $K=I$, then clearly $\norm{P^\perp \hat x}$ can be bounded in $Z = L^1$ for $\hat x$ the optimal solution to \eqref{eq:problem-rlsq}.
Thus, for some $M>0$, we can add to \eqref{eq:problem-rlsq} the artificial constraint
\begin{equation}
    \label{eq:extraconstr}
    G'(x) \defeq \delta_{\norm{\freevar}_{Z} \le M}(P^\perp x).
\end{equation}
In finite dimensions, this gives a bound in $L^2$.
Lemma \ref{lemma:bounded-strong-convexity-constant} gives \eqref{eq:g-strong-convexity-scalar-restrict} with $\bar \gamma^\perp=\infty$. 

In case of our total variation examples, $F(x)=\norm{x}_1$ and $K=\grad$. Provided mean-zero functions are not in the kernel of $A$, one can through Poincaré's inequality \cite{ambrosio2000fbv} on $\BVspace(\Omega)$ and a two-dimensional connected domain $\Omega \subset \R^2$, show that even the original infinite-dimensional problems have bounded solutions in $L^2(\Omega)$. 
We may therefore again add the artificial constraint \eqref{eq:extraconstr} with $Z=L^2$ to \eqref{eq:problem-rlsq}.

\paragraph{Dynamic bounds and pseudo duality gaps}

We seldom know the exact bound $M$, but can derive conservative estimates. Nevertheless adding such a bound to Algorithm \ref{alg:alg-proj} is a simple, easily-implemented projection of $P^\perp(x^i - \Tau_i K^* y^i)$ into the constraint set.
In practise, we do not use or need the projection, and update the bound $M$ dynamically so as to ensure that the constraint \eqref{eq:extraconstr} is never active.
Indeed, $A$ having a non-trivial nullspace also causes duality gaps for \eqref{eq:problem} to be numerically infinite. In \cite{tuomov-dtireg} a ``pseudo duality gap'' was therefore introduced, based on dynamically updating $M$. We will also use this type of dynamic duality gaps in our reporting.

\subsection{$\TGV^2$ denoising and related problem structure}

So far, we have considered very simple regularisation terms.
Total generalised variation, $\TGV$, was introduced in \cite{bredies2009tgv} as a higher-order generalisation of $\TV$. It avoids the unfortunate stair-casing effect of $\TV$---large flat areas with sharp transitions---while preserving the critical edge preservation property that smooth regularisers lack.
We concentrate on the second-order $\TGV^2$.
In all of our image processing examples, we can replace $\TV$ by $\TGV^2$.

As with total variation, we have to consider discretised models due the original problem being set in the Banach space $\BVspace(\Omega)$. For two parameters $\alpha,\beta>0$, the regularisation functional is written in the differentiation cascade form of \cite{sampta2011tgv} as
\[
    \TGV^2_{(\beta,\alpha)}(u)
    \defeq \min_w~ \alpha \norm{\grad u-w}_1 + \beta\norm{\Eabs{u}}_1.
\]
Here $\Eabs=(\grad^T+\grad)/2$ is the symmetrised gradient. With $x=(u, w)$ and $y=(y_1, y_2)$, we may write the problem
\begin{equation}
    \label{eq:tgv-reconstr}
    \min_u G_0(f-Au) + \TGV^2_{(\beta,\alpha)}(u),
\end{equation}
in the saddle-point form \eqref{eq:problem} with
\[
    G(x) \defeq G_0(f-Au), 
    \quad
    F^*(y)=\delta_{\norm{\freevar}_{L^\infty} \le \alpha}(y_1) + \delta_{\norm{\freevar}_{L^\infty} \le \beta}(y_2),
    \quad\text{and}\quad
    K \defeq
        \begin{pmatrix}
            \grad & -I \\
            0 & \Eabs
        \end{pmatrix}.
\]


If $A=I$, as is the case for denoising,
 this is an instance of the general structure
\[
    G(x_1, x_2)=G_1(x_1)+G_2(x_2), \quad
    F^*(y_1, y_2)=F^*_1(y_1)+F^*_2(y_2),
    \quad\text{and}\quad
    K \defeq
        \begin{pmatrix}
            K_{1,1} & K_{1,2} \\
            0 & K_{2,2}
        \end{pmatrix},
\]
where $G_1$ is strongly convex with factor $\gamma$.
To apply Algorithm \ref{alg:alg-projective-both}, we therefore need to find $C_\psi$ and $\bar\gamma^\perp$ satisfying for all $0 \le \gamma^\perp \le \bar\gamma^\perp$ the condition
\begin{equation}
    \label{eq:gcond-tgv}
    G_2(x_2')-G_2(x_2) \ge \iprod{\subdiff G_2(x_2)}{x_2'-x_2} + \frac{\gamma^\perp}{2}(\norm{x_2'-x_2}^2 - C_\psi),
    \quad (x_2', x_2 \in X_2).
\end{equation}
For both Algorithm \ref{alg:alg-proj} and Algorithm \ref{alg:alg-projective-both}, we also need $F_j^*$, ($j=1,2$), to satisfy for all $0 \le \rho \le \bar\rho$ and some $C_\phi$ the condition
\begin{equation}
    \label{eq:fcond-tgv}
    F_j^*(y_j')-F_j^*(y_j) \ge \iprod{\subdiff F_j^*(y_j)}{y_j'-y_j} + \frac{\rho}{2}(\norm{y_j'-y_j}^2 - C_\phi),
    \quad (y_j', y_j \in Y_j).
\end{equation}
If these conditions hold, we have
\[
    \Gamma=\gamma P
    \quad\text{for}\quad
    P=\begin{pmatrix} I & 0 \\ 0 & 0 \end{pmatrix}.
\]
As this is compatible with the splitting of $G$ into $G_1$ and $G_2$, the prox-update  \eqref{eq:alg-projective-both-nextx} splits into the uncoupled updates
\begin{align}
    \notag
    x_1^{i+1} & = \inv{(I+\tau_i G_1^*)}(x_1^i - \tau_i K_{1,1}^* y_1^i), 
    \\
    \notag
    x_2^{i+1} & = \inv{(I+\tau_i^\perp G_2^*)}(x_2^i - \tau_i^\perp K_{1,2}^* y_1^i - \tau_i^\perp K_{2,2}^* y_2^i ).
\end{align}

For the general class of problems, \eqref{eq:f-convexity-scalar-restrict} with $\bar\rho=\infty$ is immediate from Lemma \ref{lemma:bounded-strong-convexity-constant}.
For $\TGV^2$ denoising in particular, the Sobolev--Korn inequality \cite{temam1985mpp} allows us to bound on a connected domain $\Omega \subset \R^2$ an optimal $\hat w$ to \eqref{eq:tgv-reconstr} as
\[
    \inf_{\bar w \text{ affine}} \norm{\hat w-\bar w}_{L^2} \le C_\Omega \norm{\Eabs\hat w}_1 \le C_\Omega G_0(f)
\]
for some constant $C_\Omega > 0$. We may assume that $\bar w=0$, as the affine part of $w$ is not used in \eqref{eq:tgv-reconstr}.
Therefore we may again add the artificial constraint $G_2(w)=\delta_{\norm{\freevar}_{L^2} \le M}(w)$ to the $\TGV^2$ denoising problem. By Lemma \ref{lemma:bounded-strong-convexity-constant}, $G$ will then satisfy \eqref{eq:g-strong-convexity-scalar-restrict} with $\bar\gamma^\perp=\infty$.


\subsection{Numerical results}
\label{sec:numerical}

We demonstrate our algorithms on $\TGV^2$ denoising and $\TV$ deblurring. Our tests are done on the photographs in Figure \ref{fig:image}, both at the original resolution of $768 \times 512$, and scaled down by a factor of $0.25$ to $192 \times 128$ pixels. 
For both of our example problems, we calculate a target solution by taking one million iterations of the basic PDHGM \eqref{eq:cp}. We also tried interior point methods for this, but they are only practical for the smaller denoising problem.

\begin{figure}[t!]
    \centering
    \begin{subfigure}{0.45\textwidth}
        \includegraphics[width=\textwidth]{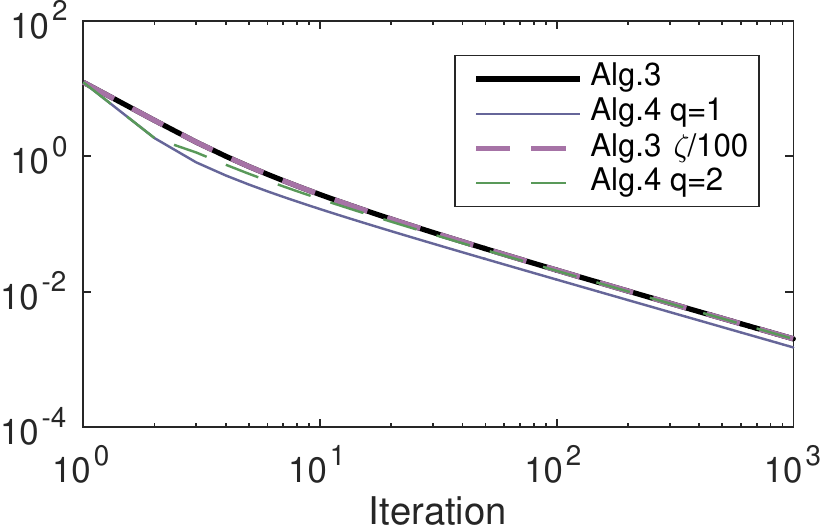}
        \caption{$\tau$}
        \label{fig:param-tau}
    \end{subfigure}
    \begin{subfigure}{0.45\textwidth}
        \includegraphics[width=\textwidth]{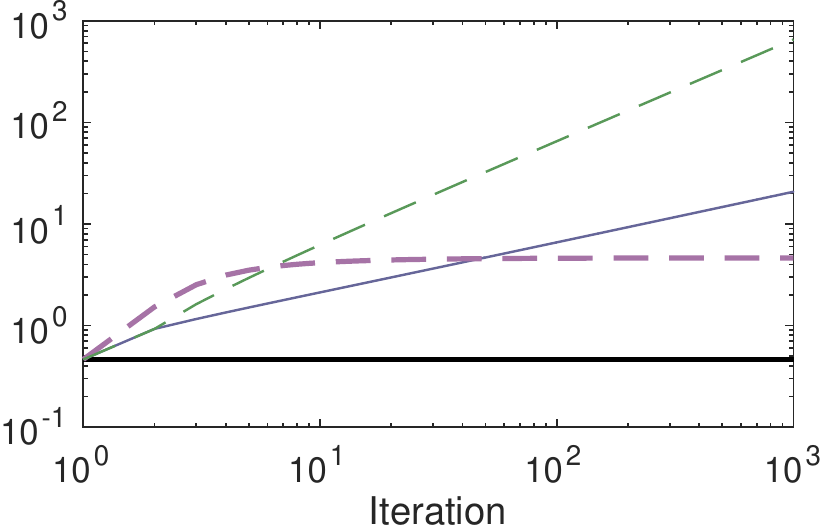}
        \caption{$\tau^\perp$}
        \label{fig:param-tau2}
    \end{subfigure}
    \begin{subfigure}{0.45\textwidth}
        \includegraphics[width=\textwidth]{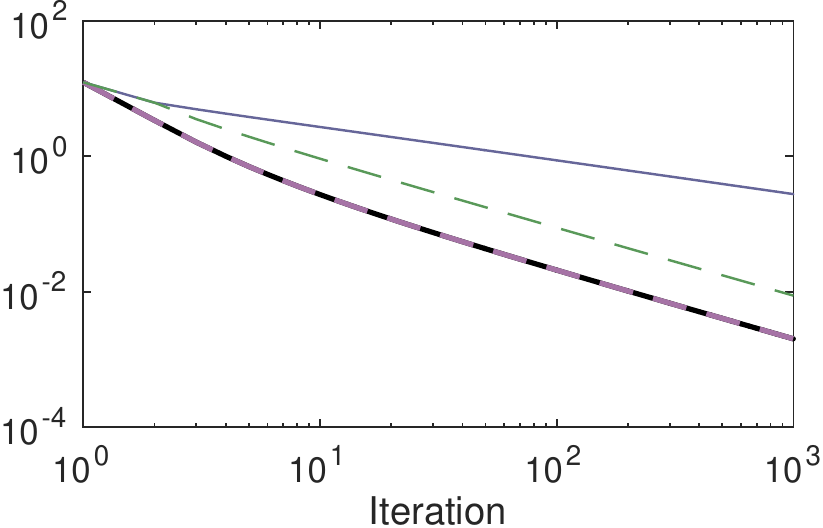}
        \caption{$\tilde\tau$}
        \label{fig:param-tautilde}
    \end{subfigure}
    \begin{subfigure}{0.45\textwidth}
        \includegraphics[width=\textwidth]{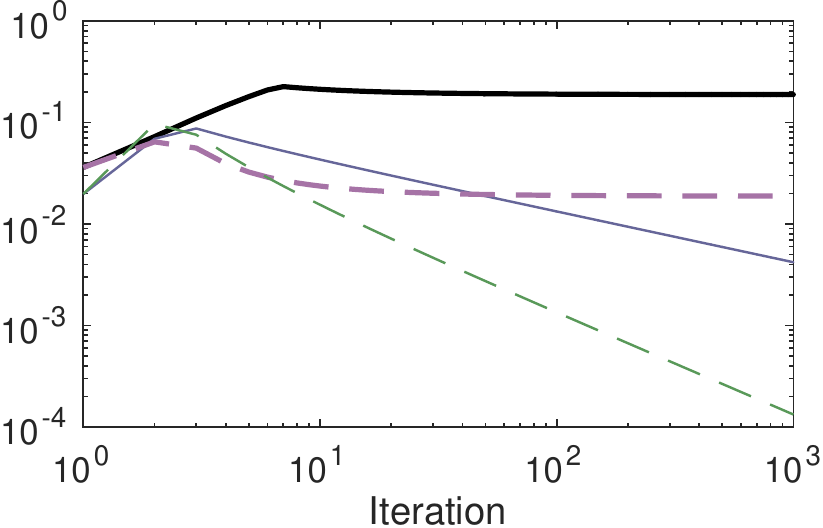}
        \caption{$\sigma$}
        \label{fig:param-sigma}
    \end{subfigure}
    \caption{Step length parameter evolution, both axes logarithmic.
        `Alg.3' and `Alg.4 q=1' have the same parameters as our numerical experiments for the respective algorithms, in particular  $\zeta=\tau_0^{\perp,-2}$ for Algorithm \ref{alg:alg-projective-both}, which yields constant $\tau^\perp$.  `Alg.3 $\zeta/100$' uses the value $\zeta=\tau_0^{\perp,-2}/100$, which causes $\tau^\perp$ to increase for some iterations. `Alg.4 q=2' uses the value $q=2$ for Algorithm \ref{alg:alg-proj}, everything else being kept equal.
        }
    \label{fig:param}
\end{figure}

We evaluate Algorithm \ref{alg:alg-projective-both} and \ref{alg:alg-proj} against the standard unaccelerated PDHGM of \cite{chambolle2010first}, as well as (a) the mixed-rate method of \cite{chen2015optimal}, denoted here C-L-O, (b) the relaxed PDHGM of \cite{chambolle2014ergodic,he2012convergence}, denoted here `Relax', and (c) the adaptive PDHGM of \cite{goldstein2015adaptive}, denoted here `Adapt'.
All of these methods are very closely linked, and have comparable low costs for each step. This makes them straightforward to compare.


As we have discussed, for comparison and stopping purposes, we need to calculate a pseudo duality gap as in \cite{tuomov-dtireg}, because the real duality gap is in practise infinite when $A$ has a non-trivial nullspace. We do this dynamically, upgrading the $M$ in \eqref{eq:extraconstr} every time we compute the duality gap.
For both of our example problems, we use for simplicity $Z=L^2$ in \eqref{eq:extraconstr}. In the calculation of the final duality gaps comparing each algorithm, we then take as $M$ the maximum over all evaluations of all the algorithms. This makes the results fully comparable. 
We always report the duality gap in decibels $10\log_{10}(\text{gap}^2/\text{gap}_0^2)$ relative to the initial iterate. Similarly, we report the distance to the target solution $\hat u$ in decibels $10\log_{10}(\norm{u^i-\hat u}^2/\norm{\hat u}^2)$,
and the primal objective value $\text{val}(x) \defeq G(x)+F(Kx)$ relative to the target as
$10\log_{10}(\text{val}(x)^2/\text{val}(\hat x)^2)$.
Our computations were performed in Matlab+C-MEX on a MacBook Pro with 16GB RAM and a 2.8 GHz Intel Core i5 CPU.

\newlength{\plotw}
\setlength{\plotw}{0.23\textwidth}

\begin{figure}[t!]
    \centering
    \begin{subfigure}{0.31\textwidth}\flushright
        \includegraphics[height=\plotw]{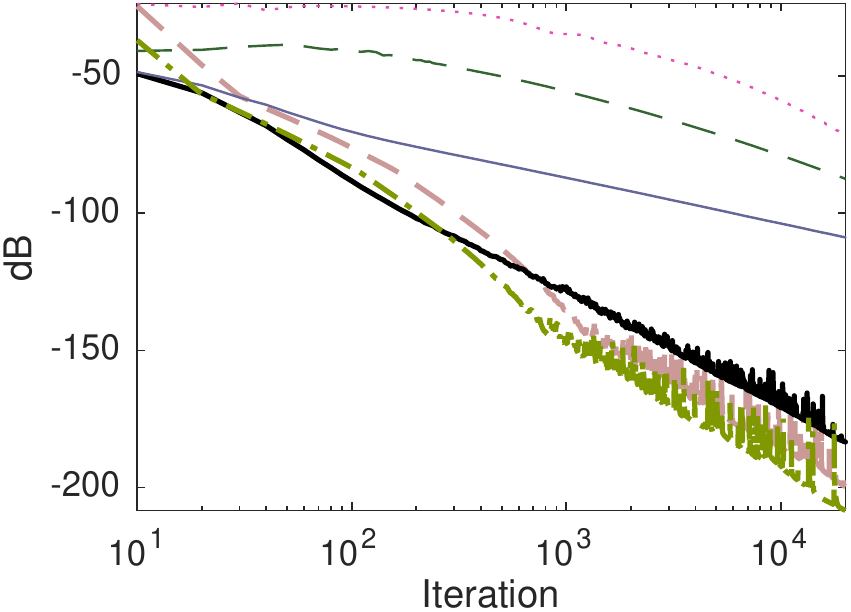}
        \caption{Gap, low resolution}
        \label{fig:tgv-denoising-gamma05-gap-low}
    \end{subfigure}
    \begin{subfigure}{0.31\textwidth}\flushright
        \includegraphics[height=\plotw]{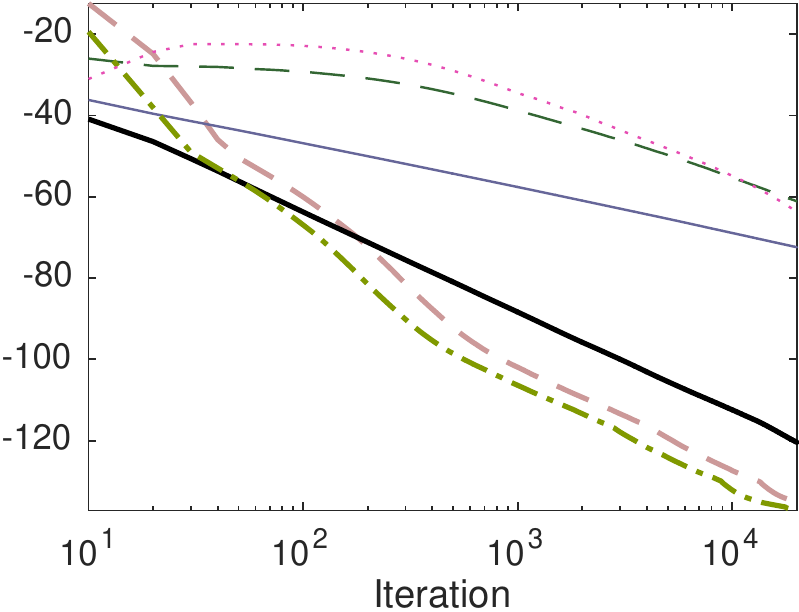}
        \caption{Target, low resolution}
        \label{fig:tgv-denoising-gamma05-target-low}
    \end{subfigure}
    \begin{subfigure}{0.31\textwidth}\flushright
        \includegraphics[height=\plotw]{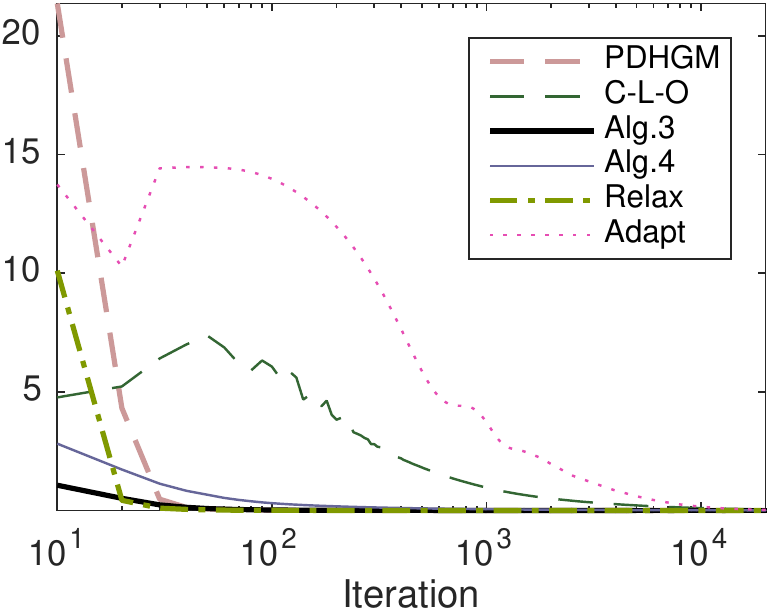}
        \caption{Value, low resolution}
        \label{fig:tgv-denoising-gamma05-val-low}
    \end{subfigure}
    \begin{subfigure}{0.31\textwidth}\flushright
        \includegraphics[height=\plotw]{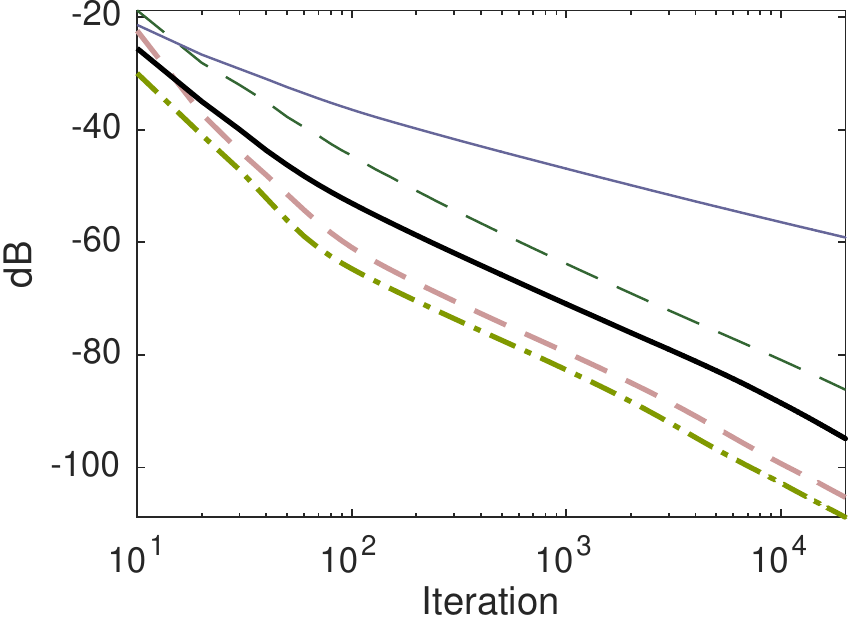}
        \caption{Gap, high resolution}
        \label{fig:tgv-denoising-gamma05-gap-high}
    \end{subfigure}
    \begin{subfigure}{0.31\textwidth}\flushright
        \includegraphics[height=\plotw]{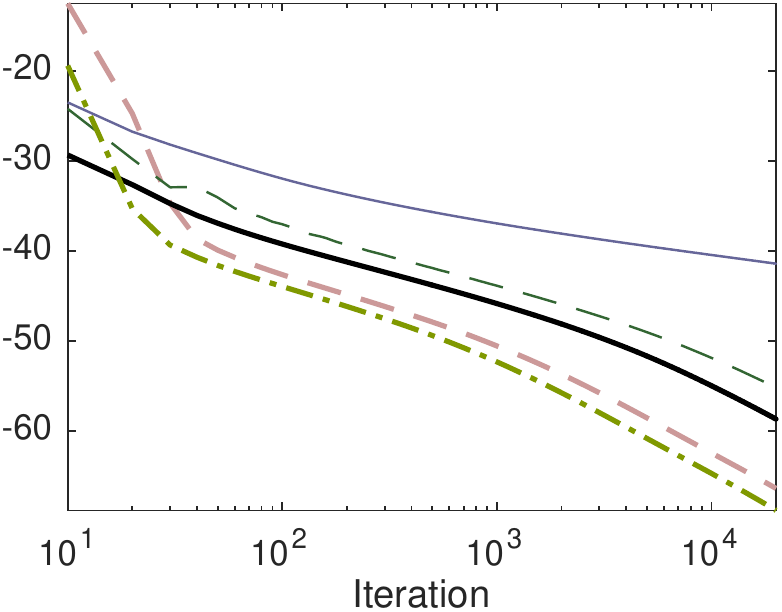}
        \caption{Target, high resolution}
        \label{fig:tgv-denoising-gamma05-target-high}
    \end{subfigure}
    \begin{subfigure}{0.31\textwidth}\flushright
        \includegraphics[height=\plotw]{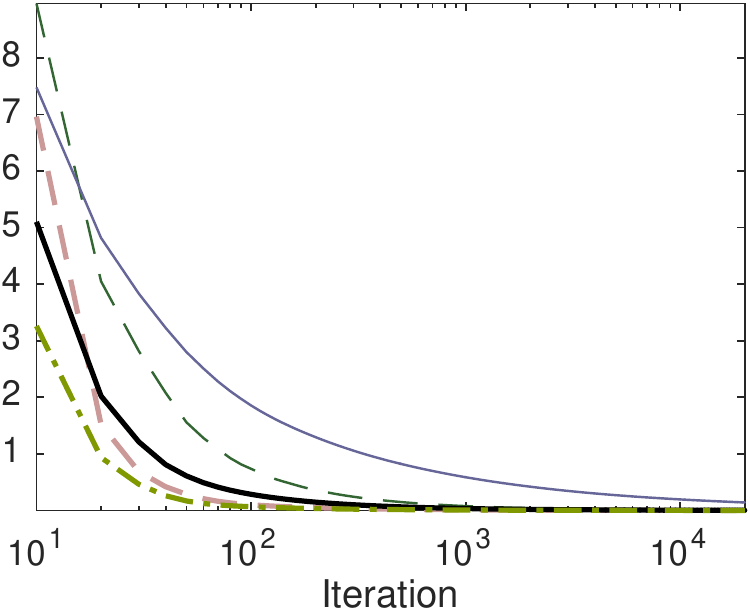}
        \caption{Value, high resolution}
        \label{fig:tgv-denoising-gamma05-val-high}
    \end{subfigure}
    \caption{$\TGV^2$ denoising performance, 20000 iterations, high and low resolution images. The plot is logarithmic, with the decibels calculated as in Section \ref{sec:numerical}. The poor high-resolution results for `Adapt' \cite{goldstein2015adaptive} have been omitted to avoid poor scaling of the plots.}
    \label{fig:tgv-denoising-gamma05}
\end{figure}

\begin{table}[t!]
    \centering
    \footnotesize
    \setlength{\tabcolsep}{3pt}
    \begin{tabular}{l|rr|rr|rr}
\multicolumn{7}{c}{high resolution}\\
\hline
 & \multicolumn{2}{c}{gap $\le -50$dB} & \multicolumn{2}{|c}{tgt $\le -50$dB} & \multicolumn{2}{|c}{val $\le 1$dB}\\
Method & iter & time & iter & time & iter & time\\
\hline
PDHGM & 30 & 0.40s & 50 & 0.53s & 30 & 0.40s\\
C-L-O & 500 & 4.67s & 5170 & 51.78s & 970 & 9.04s\\
Alg.3 & 20 & 0.29s & 30 & 0.36s & 20 & 0.29s\\
Alg.4 & 20 & 0.40s & 200 & 1.92s & 40 & 0.62s\\
Relax & 20 & 0.34s & 40 & 0.57s & 20 & 0.34s\\
Adapt & 5360 & 106.63s & 6130 & 121.98s & 3530 & 70.78s\\
\end{tabular}

    \ \ %
    \begin{tabular}{rr|rr|rr}
\multicolumn{6}{c}{high resolution}\\
\hline
 \multicolumn{2}{c}{gap $\le -50$dB} & \multicolumn{2}{|c}{tgt $\le -50$dB} & \multicolumn{2}{|c}{val $\le 1$dB}\\
iter & time & iter & time & iter & time\\
\hline
 50 & 8.85s & 870 & 128.08s & 30 & 5.13s\\
 190 & 37.47s & 6400 & 1261.36s & 80 & 15.76s\\
 80 & 12.30s & 3320 & 512.35s & 40 & 6.20s\\
 2080 & 317.93s & -- & -- & 340 & 52.06s\\
40 & 7.45s & 580 & 106.05s & 20 & 3.70s\\
\\
\end{tabular}

    \caption{$\TGV^2$ denoising performance, maximum 20000 iterations. The CPU time and number of iterations (at a resolution of 10) needed to reach given solution quality in terms of the duality gap, distance to target, or primal objective value.}
    \label{table:tgv-denoising-gamma05}
\end{table}

\paragraph{$\TGV^2$ denoising}

The noise in our high-resolution test image, with values in the range $[0, 255]$ has standard deviation $29.6$ or $12$dB. In the downscaled image, these become, respectively, $6.15$ or $25.7$dB. As parameters $(\beta, \alpha)$ of the $\TGV^2$ regularisation functional, we choose $(4.4, 4)$ for the downscale image, and translate this to the original image by multiplying by the scaling vector $(0.25^{-2}, 0.25^{-1})$ corresponding to the $0.25$ downscaling factor. See \cite{tuomov-tgvlearn} for a discussion about rescaling and regularisation factors, as well as for a justification of the $\beta/\alpha$ ratio.

For the PDHGM and our algorithms, we take $\gamma=0.5$, corresponding to the gap convergence results. 
We choose $\delta=0.01$, and parametrise the PDHGM with $\sigma_0=1.9/\norm{K}$ and $\tau_0^*=\tau_0 \approx 0.52/\norm{K}$ solved from $\tau_0\sigma_0=(1-\delta)\norm{K}^2$. These are values that typically work well. For forward-differences discretisation of $\TGV^2$ with cell width $h=1$, we have $\norm{K}^2 \le 11.4$ \cite{tuomov-dtireg}.
We use the same value of $\delta$ for Algorithm \ref{alg:alg-projective-both} and Algorithm \ref{alg:alg-proj}, but choose $\tau_0^\perp = 3\tau_0^*$, and $\tau_0=\tilde\tau_0=80\tau_0^*$. We also take $\zeta=\tau_0^{\perp,-2}$ for Algorithm \ref{alg:alg-projective-both}. These values have been found to work well by trial and error, while keeping $\delta$ comparable to the PDHGM. A similar choice of $\tau_0$ with a corresponding modification of $\sigma_0$ would significantly reduce the performance of the PDHGM. For Algorithm \ref{alg:alg-proj} we take exponent $q=1$ for the sequence $\{a_i\}$. This gives in principle a mixed $O(1/N^{1.5})+O(1/N^{0.5})$ rate, possibly improved by the convergence of the dual sequence. We plot the evolution of the step length for these and some other choices in Figure \ref{fig:param}.
For the C-L-O, we use the detailed parametrisation from \cite[Corollary 2.4]{chen2011image}, taking as $\Omega_Y$ the true $L^2$-norm Bregman divergence of
$\B(0, \alpha) \times \B(0, \beta)$, and $\Omega_X=10 \cdot \norm{f}^2/2$ as a conservative estimate of a ball containing the true solution.
For `Adapt' we use the exact choices of $\alpha_0$, $\eta$, and $c$ from \cite{goldstein2015adaptive}. For `Relax' we use the value $1.5$ for the inertial $\rho$ parameter of \cite{chambolle2014ergodic}. For both of these algorithms, we use the same choices of $\sigma_0$ and $\tau_0$ as for the PDHGM.

\begin{figure}[t!]
    \centering
    \begin{subfigure}{0.31\textwidth}\flushright
        \includegraphics[height=\plotw]{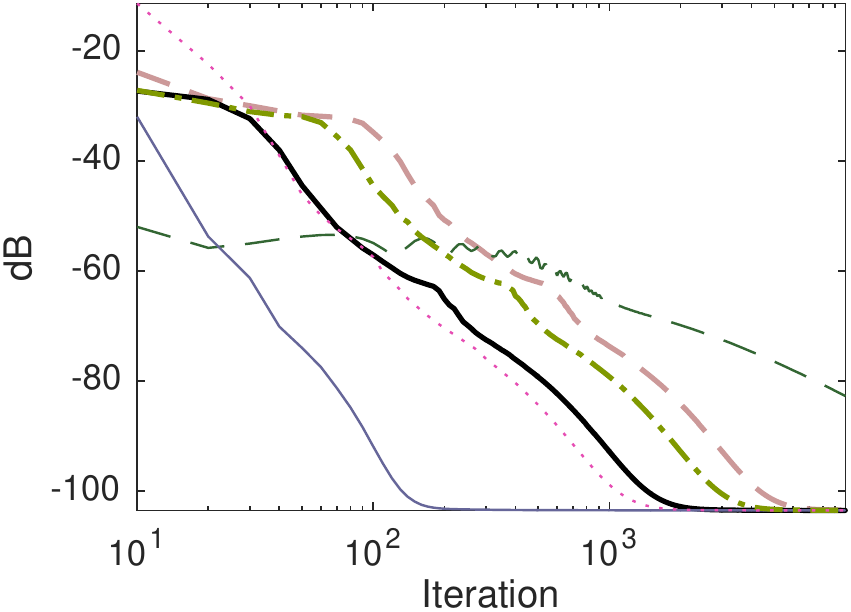}
        \caption{Gap, low resolution}\flushright
        \label{fig:tv-deblurring-gamma05-gap-low}
    \end{subfigure}
    \begin{subfigure}{0.31\textwidth}\flushright
        \includegraphics[height=\plotw]{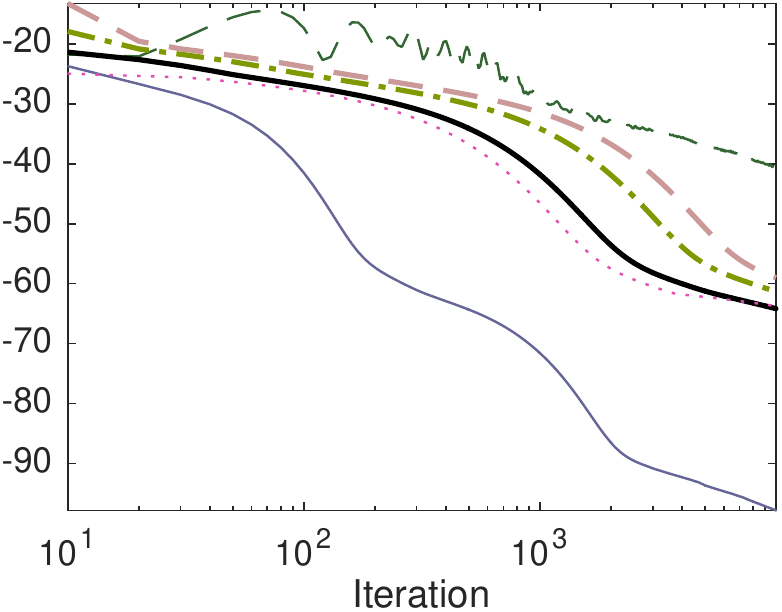}
        \caption{Target, low resolution}
        \label{fig:tv-deblurring-gamma05-target-low}
    \end{subfigure}
    \ \hfil%
    \begin{subfigure}{0.31\textwidth}\flushright
        \includegraphics[height=\plotw]{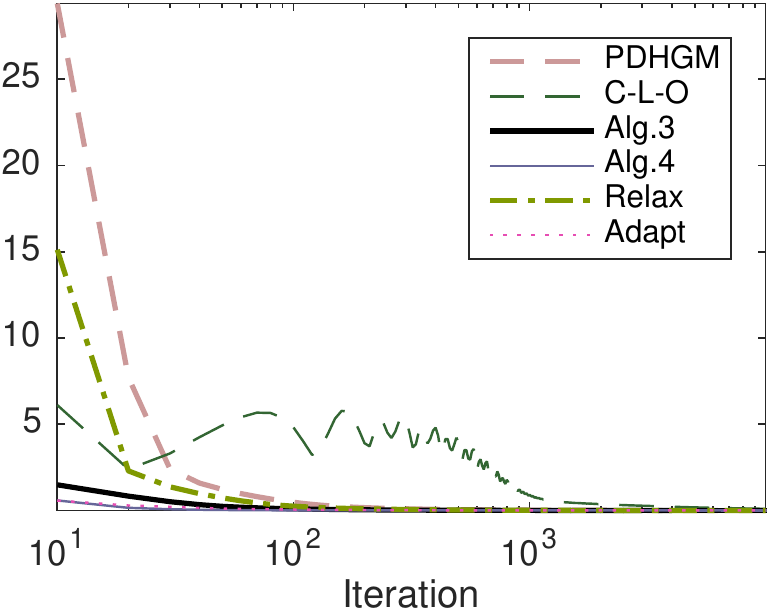}
        \caption{Value, low resolution}
        \label{fig:tv-deblurring-gamma05-value-low}
    \end{subfigure}
    \begin{subfigure}{0.31\textwidth}\flushright
        \includegraphics[height=\plotw]{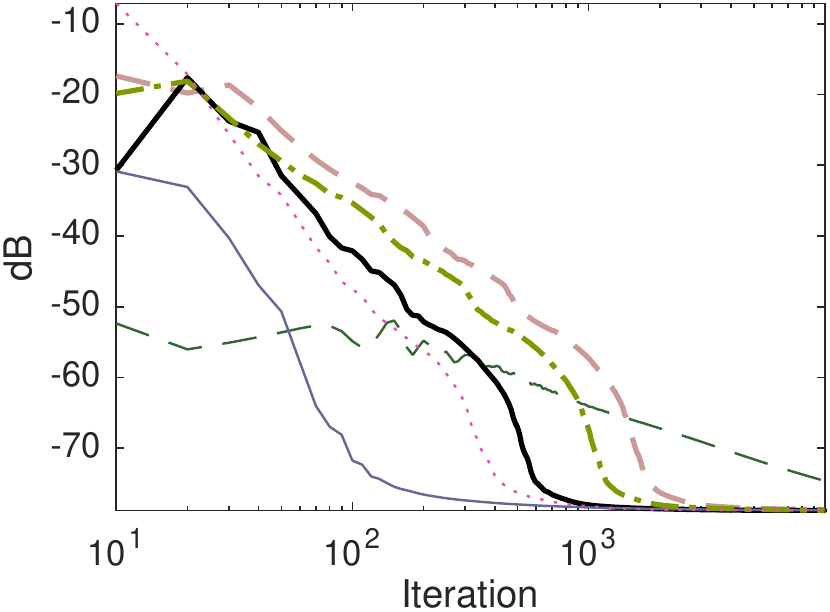}
        \caption{Gap, high resolution}
        \label{fig:tv-deblurring-gamma05-gap-high}
    \end{subfigure}
    \begin{subfigure}{0.31\textwidth}\flushright
        \includegraphics[height=\plotw]{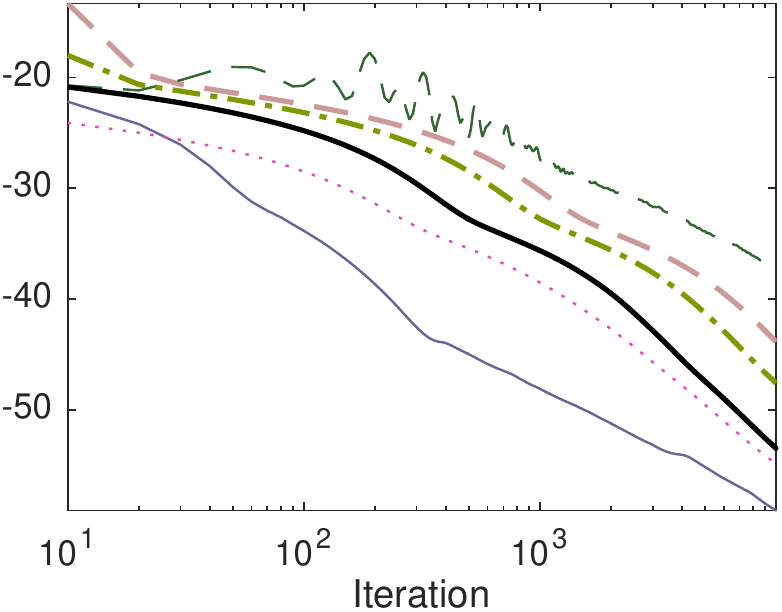}
        \caption{Target, high resolution}
        \label{fig:tv-deblurring-gamma05-target-high}
    \end{subfigure}
    \ \hfil%
    \begin{subfigure}{0.31\textwidth}\flushright
        \includegraphics[height=\plotw]{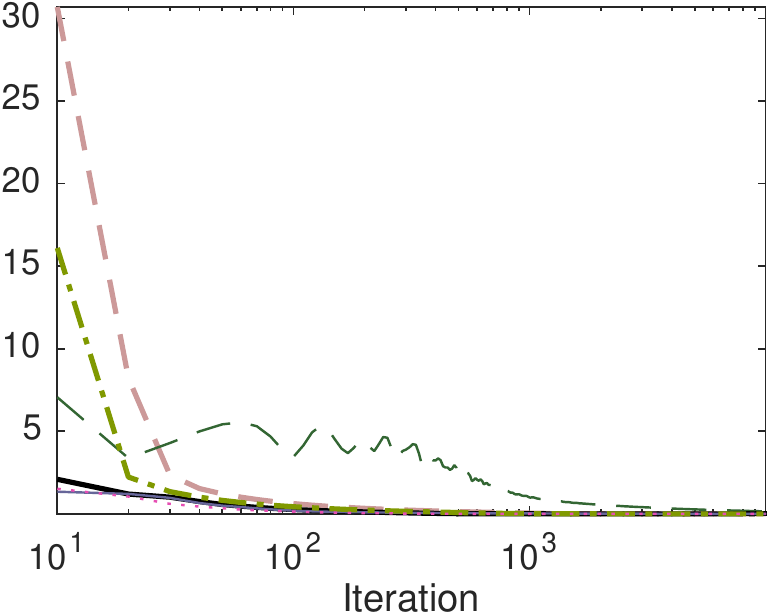}
        \caption{Value, high resolution}
        \label{fig:tv-deblurring-gamma05-value-high}
    \end{subfigure}
    \caption{$\TV$ deblurring performance, 10000 iterations, high and low resolution images. The plot is logarithmic, with the decibels calculated as in Section \ref{sec:numerical}.}
    \label{fig:tv-deblurring-gamma05}
\end{figure}

\begin{table}[t!]
    \centering
    \footnotesize
    \setlength{\tabcolsep}{3pt}
    \begin{tabular}{l|rr|rr|rr}
\multicolumn{7}{c}{high resolution}\\
\hline
 & \multicolumn{2}{c}{gap $\le -50$dB} & \multicolumn{2}{|c}{tgt $\le -50$dB} & \multicolumn{2}{|c}{val $\le 1$dB}\\
Method & iter & time & iter & time & iter & time\\
\hline
PDHGM & 200 & 1.29s & 4800 & 32.72s & 60 & 0.47s\\
C-L-O & 10 & 0.14s & -- & -- & 950 & 5.95s\\
Alg.3 & 70 & 0.62s & 1630 & 13.40s & 20 & 0.25s\\
Alg.4 & 20 & 0.29s & 140 & 1.24s & 10 & 0.22s\\
Relax & 130 & 0.85s & 3200 & 20.06s & 40 & 0.29s\\
Adapt & 70 & 0.73s & 1210 & 11.30s & 10 & 0.16s\\
\end{tabular}

    \ \ %
    \begin{tabular}{rr|rr|rr}
\multicolumn{6}{c}{high resolution}\\
\hline
 \multicolumn{2}{c}{gap $\le -50$dB} & \multicolumn{2}{|c}{tgt $\le -50$dB} & \multicolumn{2}{|c}{val $\le 1$dB}\\
iter & time & iter & time & iter & time\\
\hline
 500 & 49.84s & -- & -- & 70 & 6.59s\\
 10 & 1.05s & -- & -- & 1000 & 96.60s\\
 170 & 24.03s & 6760 & 925.94s & 40 & 6.13s\\
 50 & 6.01s & 1550 & 215.95s & 30 & 3.66s\\
 340 & 33.57s & -- & -- & 50 & 5.29s\\
 120 & 18.76s & 5300 & 800.84s & 30 & 4.72s\\
\end{tabular}

    \caption{$\TV$ deblurring performance, maximum 10000 iterations. The CPU time and number of iterations (at a resolution of 10) needed to reach given solution quality in terms of the duality gap, distance to target, or primal objective value.}
    \label{table:tv-deblurring-gamma05}
\end{table}

We take fixed 20000 iterations, and initialise each algorithm with $y^0=0$ and $x^0=0$. To reduce computational overheads, we compute the duality gap and distance to target only every 10 iterations instead of at each iteration.
The results are in Figure \ref{fig:tgv-denoising-gamma05}, and Table \ref{table:tgv-denoising-gamma05}. 
As we can see, Algorithm \ref{alg:alg-projective-both} performs extremely well for the low resolution image, especially in its initial iterations. After about 700 or 200 iterations, depending on the criterion, the standard and relaxed PDHGM start to overtake. This is a general effect that we have seen in our tests: the standard PDHGM performs in practise very well asymptotically, although in principle all that exists is a $O(1/N)$ rate on the ergodic duality gap. 
Algorithm \ref{alg:alg-proj}, by contrast, does not perform asymptotically so well.
It can be extremely fast on its initial iterations, but then quickly flattens out. 
The C-L-O surprisingly performs better on the high resolution image than on the low resolution image, where it does somewhat poorly in comparison to the other algorithms.
The adaptive PDHGM performs very poorly for $\TGV^2$ denoising, and we have indeed excluded the high-resolution results from our reports to keep the scaling of the plots informative.
Overall, Algorithm \ref{alg:alg-projective-both} gives good results fast, although the basic and relaxed PDHGM seems to perform, in practise, better asymptotically.

\paragraph{$\TV$ deblurring}

Our test image has now been distorted by Gaussian blur of kernel width $4$, which we intent to remove. 
We denote by $\hat a$ the Fourier presentation of the blur operator as discussed in Section \ref{sec:rlsq}. For numerical stability of the pseudo duality gap, we zero out small entries, replacing this $\hat a$ by $\hat a \chi_{\abs{\hat a(\freevar)} \ge \norm{\hat a}_\infty/1000}(\xi)$. Note that this is only needed for the stable computation of $G^*$ for the pseudo duality gap, to compare the algorithms; the algorithms themselves are stable without this modification.
To construct the projection operator $P$, we then set $\hat p(\xi)=\chi_{\abs{\hat a(\freevar)} \ge 0.3 \norm{\hat a}_\infty}(\xi)$, and $P=\mathcal{F}^* \hat p \mathcal{F}$.

We use $\TV$ parameter $2.55$ for the high resolution image and the scaled parameter $2.55*0.15$ for the low resolution image.
We parametrise all the algorithms is exactly as $\TGV^2$ denoising above, of course with appropriate $\Omega_U$ and $\norm{K}^2 \le 8$ for $K=\grad$ \cite{chambolle2004meanalgorithm}.

The results are Figure \ref{fig:tv-deblurring-gamma05} and Table \ref{table:tv-deblurring-gamma05}. It does not appear numerically feasible to go significantly below $-100$dB or $-80$dB gap.
Our guess is that this is due to the numerical inaccuracies of the Fast Fourier Transform implementation in Matlab.
The C-L-O performs very well judged by the duality gap, although the images themselves and the primal objective value appear to take a little bit longer to converge. 
The relaxed PDHGM is again slightly improved from the standard PDHGM. The adaptive PDHGM performs very well, slightly outperforming Algorithm \ref{alg:alg-projective-both}, although not Algorithm \ref{alg:alg-proj}.
This time Algorithm \ref{alg:alg-proj} performs exceedingly well.


\section*{Acknowledgements}

This research was started while T.~Valkonen was at the Center for Mathematical Modeling at Escuela Politécnica Nacional in Quito, supported by a Prometeo scholarship of the Senescyt (Ecuadorian Ministry of Science, Technology, Education, and Innovation). In Cambridge T. Valkonen has been supported by the EPSRC grant EP/M00483X/1 “Efficient computational tools for inverse imaging problems”.
Thomas Pock is supported by the European Research Council under the Horizon 2020 program, ERC starting grant agreement 640156.

\section*{A data statement for the EPSRC}

This is primarily a theory paper, with some demonstrations on a photograph freely available from the Internet.  As this article was written, the used photograph from the Kodak image suite was in particular available at \url{http://r0k.us/graphics/kodak/}.
It has also been archived with our implementations of the algorithms at \url{https://www.repository.cam.ac.uk/handle/1810/253697}.

 \providecommand{\homesiteprefix}{http://iki.fi/tuomov/mathematics}
  \providecommand{\eprint}[1]{\href{http://arxiv.org/abs/#1}{arXiv:#1}}
  \providecommand{\eprint}[1]{\href{http://arxiv.org/abs/#1}{arXiv:#1}}


\end{document}